\documentclass[11pt]{article}
\usepackage{amssymb,amsmath,latexsym}
\usepackage{amsthm}
\usepackage[shortlabels]{enumitem}
\usepackage{caption}
\usepackage{xcolor}
\usepackage{float}
\usepackage{graphicx}

\usepackage{url}
\usepackage{breakurl}
\usepackage[colorlinks=true,citecolor=blue,linkcolor=blue,urlcolor=blue,bookmarks,bookmarksopen,bookmarksdepth=2,backref=page,breaklinks]{hyperref}

\newenvironment{verification}{\begin{proof}}{\end{proof}}

\usepackage{bookmark}
\bookmarksetup{
	numbered, 
	open,
}

\newcommand*{\doi}[1]{\href{https://doi.org/\detokenize{#1}}{https://doi.org/\detokenize{#1}}}
\renewcommand*{\backref}[1]{}
\renewcommand*{\backrefalt}[4]{%
	\ifcase #1 (Not cited.)%
	\or        (Cited on page~#2.)%
	\else      (Cited on pages~#2.)%
	\fi}

\newcommand{\N}{\mathbb N}
\newcommand{\Z}{\mathbb Z}

\newcommand{\F}{\mathbb F}

\DeclareMathOperator{\Tr}{Tr}

\theoremstyle{plain}
\newtheorem{theorem}{Theorem}[section]
\newtheorem{lemma}[theorem]{Lemma}
\newtheorem{proposition}[theorem]{Proposition}

\newtheorem{corollary}[theorem]{Corollary}

\theoremstyle{definition}

\newtheorem{example}[theorem]{Example}
\newtheorem{openproblem}[theorem]{Open Problem}
\newtheorem{definition}[theorem]{Definition}
\newtheorem{remark}[theorem]{Remark}

\numberwithin{theorem}{section}
\numberwithin{equation}{section}
\numberwithin{table}{section}
\numberwithin{figure}{section}

\DeclareMathOperator{\image}{Im}
\begin{document}
\title{Value distributions of perfect nonlinear functions}
\author{Lukas K\"olsch \and Alexandr Polujan}

\author{
	Lukas K\"olsch \and Alexandr Polujan\\
}

\author{Lukas K\"olsch$^1$ and Alexandr Polujan$^2$ \vspace{0.4cm} \ \\
$^1$ University of South Florida \\\tt lukas.koelsch.math@gmail.com\vspace{0.4cm}\\
$^2$ Otto-von-Guericke-Universit\"{a}t,\ \\ Universit\"{a}tsplatz 2, 39106, Magdeburg, Germany\ \\ \tt alexandr.polujan@ovgu.de
}

\date{\today}
\maketitle
\abstract{
In this paper, we study the value distributions of perfect nonlinear functions, i.e., we investigate the sizes of image and preimage sets. Using purely combinatorial tools, we develop a framework that deals with perfect nonlinear functions in the most general setting, generalizing several results that were achieved under specific constraints. For the particularly interesting elementary abelian case, we derive several new strong conditions and classification results on the value distributions. Moreover, we show that most of the  classical constructions of perfect nonlinear functions have very specific value distributions, in the sense that they are almost balanced. Consequently, we completely determine the possible value distributions of vectorial Boolean bent functions with output dimension at most 4. Finally,   using the discrete Fourier transform, we show that in some cases value distributions can  be used to determine whether a given function is perfect nonlinear, or to decide whether given perfect nonlinear functions are equivalent.\\[1mm]

\noindent\textbf{Keywords:} Perfect nonlinear function, bent function, planar function, image sets, value distribution.
}

\thispagestyle{empty}

\section{Introduction}\label{section: 1 Introduction}
Let $G$ and $H$ be two additively written finite groups. A mapping $L\colon G\to H$ is called a \textit{homomorphism} if $L(x+a)-L(x)=L(a)$ for all $x, a \in G$. Homomorphisms $L\colon G\to H$  are essentially \textit{linear mappings} between  the finite groups $G$ and $H$, which can be equivalently characterized by the property
$$
|\{x \in G\colon  L(x+a)-L(x)=b\}| \in\{0, |G|\}.
$$

In this article, we consider functions $F\colon  G \rightarrow H$ which are as far as possible from all homomorphisms; such functions can be introduced with the notion of perfect nonlinearity as follows~\cite{Pott04}. A function $F\colon  G \rightarrow H$ is said to be \textit{perfect nonlinear} (or simply \textit{bent}) if
$$|\{x \in G\colon  F(x+a)-F(x)=b\}|=\frac{|G|}{|H|} \quad \mbox{holds for all } a \in G \setminus\{0\}\mbox{ and } b \in H.$$

In general, the terms ``perfect nonlinear'' and ``bent'' are considered to be synonymous. However, in this paper, we will use the term ``bent'' for mappings between two elementary abelian groups. Bent functions considered in this setting play a very important role in finite geometry (they give rise to commutative semifields~\cite{golouglu2022exponential}), combinatorics (one can use them to construct skew Hadamard difference sets~\cite{ding2006family}), and applications due to their rich connections to coding theory and cryptography~\cite{Carlet2021_Book,Mesnager2016}. 

\subsection{Preliminaries}
Let $G$ and $H$ be two finite groups and let $F\colon G\to H $ be a function. For an element $\beta \in H$, we denote by $F^{-1}(\beta)$ the \textit{preimage set} of $\beta$. By \textit{value distribution} of the function $F\colon G\to H $ we understand the multiset
$\{*\, |F^{-1}(\beta)| \colon \beta \in H \, *\}$. We use no special notation for this multiset, since determining the value distribution just boils down to determining the sizes of all preimages.

In the following, we will frequently consider functions $F\colon G\to H$, where $G$ and $H$ are two elementary abelian groups. In this case, we use the notation $G=\F_p^n$ and $H=\F_p^m$, where $\F_p$ is the finite field with $p$ elements and $\F_p^n$ is the vector space of dimension $n$ over the prime field $\F_p$. For $x=(x_1,\ldots,x_n),y=(y_1,\ldots,y_n)\in\F_p^n$, we define the scalar product of $\F_p^n$ by $\langle x,y\rangle_n=x_1y_1+\cdots+x_ny_n$. If necessary, we endow the vector space $\F_p^n$ with the structure of the finite field $\F_{p^{n}}$; in this case, we define the scalar product of $\F_{p^n}$ by $\langle x,y\rangle_n=\Tr(xy)$, where $\Tr(z):=\Tr_1^n(z)$ is the absolute trace and  $\Tr^n_m(z)=\sum_{i=0}^{\frac{n}{m}-1} z^{p^{i\cdot m}}$ is the relative trace of $z \in \mathbb{F}_{p^n}$ from $\mathbb{F}_{p^n}$ into the subfield $\mathbb{F}_{p^m}$. 
If $n=2k$ is even, the vector space $\F_p^n$ can be identified with $\mathbb{F}_{p^k} \times \mathbb{F}_{p^k}$; in this case, we define the scalar product $\left\langle\left(u_1, u_2\right),\left(v_1, v_2\right)\right\rangle_n=\operatorname{Tr}_1^k\left(u_1 v_1+u_2 v_2\right)$.

For an odd prime $p$, the mappings $F\colon\F_p^n\rightarrow\F_p$ are called \textit{$p$-ary functions}, and for $p=2$, \textit{Boolean functions}. For $m\geq 2$, the mappings $F\colon\F_p^n\rightarrow\F_p^m$ are called \textit{vectorial functions}.  Any vectorial function $F\colon\F_p^n\rightarrow\F_p^m$ can be uniquely described by $m$ \textit{coordinate functions} $f_i\colon\F_p^n\rightarrow\F_p$ for $1\le i\le m$ as a vector $F(x):=(f_1(x),\ldots,f_m(x))$. For $b\in\F_p^m$, the function $F_b(x):=\langle b,F(x)\rangle_m$ is called a \textit{component function} of $F$.

Vectorial and $p$-ary functions can be also represented with a help of multivariate polynomials in the ring $\F_p[x_1,\dots,x_n]/(x_1- x_1^p,\dots,x_n- x_n^p)$. This representation is unique and called the \textit{algebraic normal form} (\textit{ANF}, for short), namely for $p$-ary functions $f\colon\F_p^n\to\F_p$ it is formally defined as $f(x)=\sum_{a\in\F_p^n}c_{a}\left( \prod_{i=1}^{n} x_i^{a_i} \right)$, where $x = (x_1,\dots, x_n)\in\F_p^n$,  $c_{a}\in\F_p$ and $a = (a_1,\dots, a_n)\in\F_p^n$, while for vectorial functions $F\colon\F_p^n\to\F_p^m$ it is defined coordinate-wise. Besides the coordinate representation and algebraic normal form, we will also consider trace representations. Identifying $\F_p^n$ with $\F_{p^n}$, we can uniquely represent any function $F\colon\F_p^n\to\F_p^n$  as a polynomial $F\colon\F_{p^n}\to\F_{p^n}$ of the form $F(x)=\sum_{i=0}^{p^{n}-1} a_{i} x^{i}$ with coefficients $a_i\in\F_{p^{n}}$. Clearly, when $m|n$, any function $F\colon\F_p^n\to\F_p^m$ can be written as a polynomial $F\colon\F_{p^n}\to\F_{p^m}$ given by $F(x)=\Tr^{n}_{m}\left(\sum_{i=0}^{p^{n}-1} a_{i} x^{i}\right)$. This representation is called the \textit{univariate (trace) representation}, however, it is not unique in general. 


Now, we define the following equivalence relation, which preserves the nonlinearity of functions on elementary abelian groups. Functions $F,F'\colon\F_p^n\to \F_p^m$, are called \textit{equivalent}  (\textit{extended-affine equivalent}, to be more precise), if $F' = A_1 \circ F \circ A_2 + A$ for some affine permutations $A_1$, $A_2$ and an affine mapping $A$.  Clearly, for affine permutations $A_1$ and $A_2$, the functions $F' = A_1 \circ F \circ A_2$ and $F$ have the same value distributions, while the functions $F'=F+A $ and $F$, where $A$ is an affine mapping, generally do not have the same value distributions; the latter will be illustrated with extensive examples in the following sections. 

Our main tool for dealing with perfect nonlinear functions defined on elementary abelian groups is the \textit{discrete Fourier transform}. In this specific setting, it is often called the \textit{Walsh transform}, which is the term we will use throughout the paper. 
For a $p$-ary function $f\colon\F_p^n\to\F_p$, the Walsh transform is the complex-valued function $W_f\colon\F_{p}^n\to\mathbb{C}$ defined by
$$
W_f(b)=\sum_{x \in \F_p^n} \zeta_p^{f(x) -\langle b, x\rangle_n}, \quad\mbox{where } \zeta_p=e^{2 \pi i / p} \quad\mbox{and }i^2=-1.
$$

For vectorial functions $F\colon\F_p^n\to\F_p^m$, the Walsh transform is defined using the notion of component functions as $W_F(b,a)=W_{F_b}(a)$ for all $a\in\F_p^n,b\in\F_p^m$.
\subsection{Value distributions of bent functions: the known cases}
With the Walsh transform, bent functions can be equivalently defined in the following way, for details we refer to~\cite{KUMAR1985,Meier90,Nyberg91}.

\begin{definition}
    A function $f\colon\F_p^n\to \F_p$ is called a \textit{bent} function, if the Walsh transform  satisfies $|W_f(b)|=p^{n / 2}$ for all $b \in \F_p^n$.
\end{definition}

First, we consider in detail the Walsh transform of single-output bent functions. In the Boolean case, i.e., $p=2$ we have that $\zeta_2=-1$, from what follows that $W_f(b)$ is an integer. Consequently, for every $b\in\F_2^n$ a Boolean bent function $f$ on $\F_2^n$ satisfies $W_f(b)=2^{n / 2}(-1)^{f^*(b)}$, where $f^*\colon \F_2^n\to \F_2$ is called the \textit{dual} of $f$, what implies that $n$ must be even. The dual function $f^*$ is bent~\cite{ROTHAUS:1976}, moreover the equality $(f^*)^*=f$ holds. In the $p$ odd case, the Walsh transform $W_f(a)$ at $b \in \F_p^n$ of a $p$-ary bent function $f\colon\F_p^n\to\F_p$ satisfies~\cite{KUMAR1985}:
$$
W_f(a)=\left\{\begin{array}{cc}
\pm \zeta_p^{f^*(a)} p^{n / 2} &\mbox{ if }  p^n \equiv 1 \bmod 4 \\
\pm i \zeta_p^{f^*(a)} p^{n / 2} & \mbox{ if } p^n \equiv 3 \bmod 4
\end{array}\right.,
$$
where $f^*\colon\F_p^n\to\F_p$, is called the \textit{dual} of $f$. Opposite to the Boolean case, $p$-ary bent functions  $f\colon\F_p^n\to\F_p$ exist for all integers $n\in\N$, however, the dual of a $p$-ary bent function is not necessarily bent. A bent function
$f\colon\F_p^n\to\F_p$ is called \textit{dual-bent} if the dual $f^*$ is bent as well, otherwise, it is called \textit{non-dual-bent}. Consider the following important classes of dual-bent functions, namely weakly regular and regular bent functions. 
A bent function $f\colon\F_p^n\to\F_p$ is called \textit{weakly regular} if for all $a \in \F_p^n$, we have $W_f(a)=\epsilon \zeta_p^{f^*(a)} p^{n / 2}$ for some fixed $\epsilon \in\{\pm 1, \pm i\}$. If $\epsilon=1$, a bent function $f$ is called \textit{regular}. If no such a fixed $\epsilon\in\{\pm 1, \pm i\}$ exists, then $f$ is called \textit{non-weakly regular} bent; such functions can be either dual-bent or non-dual-bent. For further references on $p$-ary bent functions and their duals, we refer to~\cite{Meidl2022}. For the sake of simplicity, we will include Boolean functions when talking about regular functions from $\F_{p}^n$ to $\F_{p}^m$.

With the notion of component functions, vectorial bent functions can be defined in the following way~\cite{KUMAR1985,Meier90,Nyberg91}.

\begin{definition}
    A function $F\colon\F_p^n\to \F_p^m$ is called a \textit{vectorial bent} function, if for all $b \in \F_p^m\setminus\{0\}$ the component function $F_b\colon\F_p^n\to \F_p$ is bent.
\end{definition}
For vectorial Boolean bent functions $F\colon\F_2^n\to \F_2^m$, we have necessarily $m\le n/2$ (this fact is also known as the \textit{Nyberg's bound}, see~\cite{Nyberg91}), while for $p$-ary vectorial bent functions $F\colon\F_p^n\to \F_p^m$, it is possible that $n=m$; in this case bent functions $F\colon\F_p^n\to \F_p^n$ are called \textit{planar}. For a survey on bent and planar functions, we refer to~\cite{Pott2016}.

Note that bent functions belong to a larger class of plateaued functions. A function $F\colon\F_p^n\to\F_p^m$ is called \textit{plateaued}, if for every $b\in\F_p^m\setminus\{0\}$ the Walsh transform of $F_b$ at $a\in\F_p^n$ satisfies $\left|W_{F_b}(a)\right| \in\left\{0, p^{(n+s_b) / 2}\right\}$ for an integer $s_b$ with $0 \leq s_b \leq n$. Bent functions $F\colon\F_p^n\to\F_p^m$ are exactly 0-plateaued functions, i.e., $s_b=0$ for all $b\in\F_p^m\setminus\{0\}$.

Now, we survey the known results about value distributions of bent functions. The sizes of the preimage sets of Boolean bent functions were determined by Dillon in his thesis~\cite{Dillon1974}, whereas the case of $p$-ary bent functions was addressed by Nyberg~\cite{Nyberg91}. 

\begin{theorem}\cite[Theorems 3.2-3.5]{Nyberg91}\label{th: Nyberg's p-ary distribution}
    Let $p$ be a prime and $f\colon \F_p^n \to \F_p$ be a bent function, and for $l \in \F_p$, let $b_{l}=\left|f^{-1}(l)\right|$, where $f^{-1}(l)=\left\{x \in \F_p^n: f(x)=l \right\}$.
    \begin{itemize}
        \item[\textit{(i)}] If $n$ is even, then there exists a unique $c \in \mathbb{F}_p$ such that
        \begin{equation}
            \begin{split}
                b_c=&p^{n-1} \pm(p-1) p^{\frac{n}{2}-1},\\
                b_{l}=&p^{n-1} \mp p^{\frac{n}{2}-1} \quad \mbox{ for all } l \in \mathbb{F}_p \backslash\{c\}
            \end{split}
        \end{equation}
        Moreover, a regular bent function has the upper signs.
        \item[\textit{(ii)}] If $p$ and $n$ are odd, then the value distribution of a regular bent function is given by $\left(b_0, b_1, \ldots, b_{p-1}\right)$ or a cyclic shift of $\left(b_0, b_1, \ldots, b_{p-1}\right)$, where $b_0=p^{n-1}$ and
        \begin{equation} \label{eq:nyberg_odd}
            \begin{split}
                b_{l}=&p^{n-1}+\left(\frac{l}{p}\right) p^{\frac{n-1}{2}}\mbox{ for all } l \in \F_p \setminus\{0\},\mbox{ or}\\ b_{l}=&p^{n-1}-\left(\frac{l}{p}\right) p^{\frac{n-1}{2}}\mbox{ for all } l \in \F_p \setminus \{0\},
            \end{split}
        \end{equation}
         and $${\displaystyle \left(\frac{l}{p}\right)={\begin{cases}1&{\text{if }}l{\text{ is a quadratic residue modulo }}p{\text{ and }}l\not \equiv 0{\pmod {p}}\\-1&{\text{if }}l{\text{ is a non-quadratic residue modulo }}p\\0&{\text{if }}l\equiv 0{\pmod {p}}\end{cases}}}$$ is the Legendre symbol.
    \end{itemize}
\end{theorem}

Value distributions of vectorial bent functions were considered mostly for the classes of bent functions with certain prescribed properties. For instance, Nyberg~\cite[Theorem 3.2]{KaiNy91} proved that for a bent function $F\colon\F_p^n\to\F_p^m$ that has only regular (non-zero) component functions, all preimage set sizes are divisible by $p^{n/2-m}$ and derived both lower and upper bounds on preimage set sizes in this setting.  Recently, preimage sets of vectorial bent functions attracted a lot of attention due to the connection with partial difference sets observed in~\cite{CMP2021}. For instance in~\cite{CMP2021,Wang2023}, the value distributions of bent functions $F\colon\F_p^n\to\F_p^m$ with the following properties have been considered:
\begin{itemize}
    \item \textit{$l$-forms}, i.e., $F$ satisfies $F(\alpha x) = \alpha^l F(x)$ for all $\alpha \in\F_p^m$ and some fixed integer $l$ with $\gcd(p^m-1, l-1) = 1$, and,
    \item \textit{vectorial dual-bent} functions, i.e., the set of the dual functions of the component functions of $F$ together with the zero function forms a vector space of bent functions dimension $m$.
\end{itemize} 
Particularly, in~\cite[Corollary 1]{Wang2023}, it was shown that a vectorial dual-bent function $F\colon \F_p^n \rightarrow \mathbb{F}_{p}^{m}$, which satisfies $F(0)=0, F(-x)=F(x)$ and all component functions are regular (in this case, $\varepsilon=+1$) or weakly regular but not regular (in this case, $\varepsilon=-1$) satisfies
\begin{equation} \label{eq:wang}
    \left|F^{-1}(0)\right|=p^{n-m}+\varepsilon\left(p^m-1\right) p^{\frac{n}{2}-m}\mbox{ and }\left|F^{-1}(\beta)\right|=p^{n-m}-\varepsilon p^{\frac{n}{2}-m}, \mbox{ for } \beta\in \F_p^m\setminus\{0\} .
\end{equation}

Finally, in the case $n=m$, it was shown in~\cite[Theorem 2]{KyureghyanP2008} that planar functions $F\colon\F_p^n\to\F_p^n$ with the minimal image set, i.e., $|\operatorname{Im}(F)|=(p^n+1)/2$, have special value distributions, namely, they are \textit{2-to-1 mappings}.

As these results show, the value distributions of bent functions $F\colon\F_p^n\to\F_p^m$ are well-understood in the extremal cases, namely in the single-output case $m=1$ as well as in the planar case $m=n$, for $p$ odd. However, the knowledge of value distributions in the ``in-between'' cases $1<m<n$, is limited to the bent functions with specific additional properties (e.g., vectorial dual bent, $l$-forms). Moreover, the non-elementary abelian case has not been addressed at all.

In this paper, we develop a purely combinatorial general framework for the study of value distributions of perfect nonlinear functions. With our approach, we are able to unify the known results on value distributions of bent functions in different settings, which were previously obtained with different techniques. In the process, we strengthen many known results and prove new structural properties of functions with specific value distributions. Moreover, we show that our framework is also applicable for perfect nonlinear functions defined on non-elementary abelian groups.

The rest of the paper is organized in the following way. In Section~\ref{section: 2 Bounds on the cardinality of preimage sets}, we derive general, sharp upper and lower bounds on the cardinalities of the preimage sets of perfect nonlinear functions on arbitrary groups (Theorem~\ref{thm:CS general}). We show that for functions meeting these bounds, all but one values in the image set are equally distributed between preimages. Additionally, we investigate the surjectivity of perfect nonlinear functions.  In Section~\ref{section: 3 Almost balanced vectorial bent functions}, we introduce the notion of almost balanced perfect nonlinear functions; these are the perfect nonlinear functions which achieve upper/lower bounds on the cardinalities of the preimage sets with equality and thus are extremal objects of particular interest. Considering the elementary abelian framework, we show in Subsection~\ref{subection: 3.1 Primary constructions} that many primary constructions of bent functions are almost balanced. In Subsection~\ref{subection: 3.2 Secondary constructions}, we demonstrate how one can construct new almost balanced bent functions from known ones using secondary constructions. In particular, combining primary and secondary constructions, we are able to show that almost balanced bent functions exist for all admissable elementary abelian groups (Theorem~\ref{th:existence}). In Section~\ref{section: 4 Preimage sets and the Walsh transform}, we study the connection between value distributions and the Walsh transform of bent functions. Using these spectral properties, we generalize Nyberg's result on the possible  sizes of preimage sets of bent functions, giving stronger and more general conditions on preimage set sizes in both the Boolean case as well as the general $p$-ary case (Theorems~\ref{thm:regular_constraints}, \ref{thm:boolean_constraints}, \ref{thm:m odd}). We are also able to prove that in some cases, knowing the value distribution of two vectorial bent functions is enough to settle the (in general difficult) equivalence question (Corollary~\ref{cor:regular}). In Section~\ref{section: 5 Boolean (n,2)-bent functions are almost balanced}, we determine possible value distributions for bent functions with small output groups. In particular, we give a complete characterization of all possible value distributions for Boolean bent functions with output dimension at most $4$. In Section~\ref{section: 6 Preimage sets of planar functions}, we consider planar functions. Using the techniques developed in this paper, we unify several known results on the characterization of planar functions with extremal value distributions and give a more precise description of planar functions with the maximum possible image set size. Finally, we provide  new characterizations of planar functions of special shapes, again generalizing several well-known results. For instance, we are able to show that plateaued $2$-to-$1$ functions are automatically planar (Theorem~\ref{thm:planarplateaued}). In Section~\ref{section: 7 Conclusion and open problems}, we conclude the paper and give a list of open problems on perfect nonlinear functions and their value distributions.

\section{Bounds on the cardinality of preimage sets}\label{section: 2 Bounds on the cardinality of preimage sets}
In this section, we derive upper and lower bounds on the cardinalities of the preimage sets of perfect nonlinear functions on arbitrary groups and show that in the cases when the bounds are attained, we have that all but one values are equally distributed. We begin with the following simple result.  {It has already appeared in~\cite{carletding}; we add a short proof for the convenience of the reader.}
\begin{proposition} \label{prop:CS_1}
	 Let $G$ and $H$ be two finite groups, and let  $F \colon G \rightarrow H$ be a perfect nonlinear function. Then the following holds
	\[\sum_{\beta\in H}|F^{-1}(\beta)|^2 = |G|+\frac{|G|}{|H|}(|G|-1).\]
\end{proposition}
\begin{proof}
	We have $\sum_{\beta\in H}|F^{-1}(\beta)|^2=|\{(x,y) \in G\times G \colon F(x)=F(y)\}|$. Observe that 
	\[|\{(x,y) \in G\times G \colon F(x)=F(y)\}|=|G|+|\{(x,a) \in G\times (G\setminus\{0\}) \colon F(x)=F(x+a)\}|.\]
	Since $F$ is perfect nonlinear, we have that $F(x)=F(x+a)$ holds for a fixed value $a\neq 0$  for exactly $|G|/|H|$ values of $x$. In this way, $|\{(x,a) \in G\times (G\setminus\{0\}) \colon F(x)=F(x+a)\}|=|G|/|H|\cdot(|G|-1)$ and the result follows.
\end{proof}
This result can be applied to get minimum and maximum sizes of preimage set sizes of perfect nonlinear functions. For the sake of brevity, denote for a function $F \colon G \rightarrow H$ the preimage set sizes by $X_1,X_2,\dots,X_{|H|}$, where we use an arbitrary ordering. By Proposition~\ref{prop:CS_1}, for a perfect nonlinear function, we get
\begin{align}
	\sum_{i=1}^{|H|} X_i^2&=|G|+\frac{|G|}{|H|}(|G|-1), \label{eq:1}\\
	\sum_{i=1}^{|H|} X_i&=|G|,\label{eq:2}
\end{align}
	where the second equation follows  from the fact that all preimages exhaust $G$. We will now look for bounds and explicit solutions of the $X_i$.
	\begin{remark}
		Note that not every solution to Equations~\eqref{eq:1} and~\eqref{eq:2} yields a preimage distribution of a perfect nonlinear function. For instance, there is no vectorial bent function from $G=\F_2^4$ to $H=\F_2^3$ (since Nyberg's bound is violated) but for $|G|=16$ and $|H|=8$ a solution to Equations~\eqref{eq:1} and~\eqref{eq:2} exists, for example $X_1=5, X_2=3, X_3=X_4=2, X_5=\dots=X_8=1$.
	\end{remark}
	
	 {
Considering Equations~\eqref{eq:1} and~\eqref{eq:2}, it is clear that if $X_1,\dots,X_{p^m}$ are a solution, the average value of $X_i$ is $p^{n-m}$. As the next proposition shows, handling these equations is made a lot easier when one considers the deviations from this mean instead of the $X_i$ directly.
\begin{proposition}
	Define $H_i=X_i-\frac{|G|}{|H|}$. Then Equations~\eqref{eq:1} and~\eqref{eq:2} are satisfied if and only if
	\begin{align}
		\sum_{i=1}^{|H|} H_i^2&=|G|-\frac{|G|}{|H|} \label{eq:3} \\
			\sum_{i=1}^{|H|} H_i&=0. \label{eq:4}
	\end{align}
\end{proposition}
\begin{proof}
	Equation~\eqref{eq:2} is clearly equivalent to Equation~\eqref{eq:4}. For Equation~\eqref{eq:1}, we have 
	\begin{align*}
	|G|+\frac{|G|}{|H|}(|G|-1) &=\sum_{i=1}^{|H|} X_i^2= \sum_{i=1}^{|H|} \left(\frac{|G|}{|H|}+H_i\right)^2 \\
	&=\frac{|G|^2}{|H|}+2\frac{|G|}{|H|}\sum_{i=1}^{|H|}H_i+\sum_{i=1}^{|H|} H_i^2 \\
	&= \frac{|G|^2}{|H|} + \sum_{i=1}^{|H|} H_i^2.
	\end{align*}
	Rearranging yields $\sum_{i=1}^{|H|} H_i^2=|G|-\frac{|G|}{|H|}$ as desired.
\end{proof}
Note that all solutions of Equations~\eqref{eq:3} and~\eqref{eq:4} come in pairs since one can change the signs of all the $H_i$.}

 The following theorem gives general bounds on the minimum and maximum preimage set sizes of perfect nonlinear functions in the most general setting. We will see later that the bounds achieved here are (at least for elementary abelian groups) sharp.  {Note that here and in other proofs later, we will repeatedly use the second moment method.}
	\begin{theorem} \label{thm:CS general}
		 Let $G$ and $H$ be two finite groups, and let $F\colon G \rightarrow H$ be a perfect nonlinear function. Then for every $\beta \in H$ the following inequality holds
		\begin{equation}\label{eq: preimage bound general}
		    \frac{|G|}{|H|}-\sqrt{|G|}+\frac{\sqrt{|G|}}{|H|}\le |F^{-1}(\beta)| \le \frac{|G|}{|H|}+\sqrt{|G|}-\frac{\sqrt{|G|}}{|H|}.
		\end{equation}
		1. If $\displaystyle |F^{-1}(\alpha)|=\frac{|G|}{|H|}-\sqrt{|G|}+\frac{\sqrt{|G|}}{|H|}$ then $\displaystyle |F^{-1}(\beta)|=\frac{|G|}{|H|}+\frac{\sqrt{|G|}}{|H|}$ for each $\beta \neq \alpha$. \\ 2. If $\displaystyle|F^{-1}(\alpha)|=\frac{|G|}{|H|}+\sqrt{|G|}-\frac{\sqrt{|G|}}{|H|}$ then $\displaystyle|F^{-1}(\beta)|=\frac{|G|}{|H|}-\frac{\sqrt{|G|}}{|H|}$ for each $\beta \neq \alpha$.\ \\
		If the equality takes place, then $|H|$ divides $\sqrt{|G|}$, and consequently $|G|$ is a square.   
	\end{theorem}
	\begin{proof}
	 {
	We consider Equations~\eqref{eq:3} and~\eqref{eq:4}. 
		By the Cauchy-Schwarz inequality, we have 
		\[\sum_{i=2}^{|H|} H_i^2 \geq  \left(\sum_{i=2}^{|H|} H_i\right)^2 \cdot \frac{1}{|H|-1},\]
		with equality if and only if all $X_i$, $i>1$ are identical. Then
		\[|G|-\frac{|G|}{|H|}=\sum_{i=1}^{|H|} H_i^2 =H_1^2+\sum_{i=2}^{|H|} H_i^2 \geq H_1^2+\frac{H_1^2}{|H|-1}.\]
		This inequality is quadratic in $H_1$ and can be solved with elementary techniques, the result is
		\[-\sqrt{|G|}+\frac{\sqrt{|G|}}{|H|}\le H_1 \le \sqrt{|G|}-\frac{\sqrt{|G|}}{|H|},\]
		which leads for the preimage set sizes $X_i$ to the distributions
		\[\frac{|G|}{|H|}-\sqrt{|G|}+\frac{\sqrt{|G|}}{|H|}\le X_1 \le \frac{|G|}{|H|}+\sqrt{|G|}-\frac{\sqrt{|G|}}{|H|}.\]
		In the extremal cases, equality in the Cauchy-Schwarz inequality has to hold, so all $H_i$, $i>1$ are identical (and thus also the $X_i$) and we get $$X_i=\frac{|G|-X_1}{|H|-1}\quad \mbox{for all } i>1.$$ The result follows by plugging in the extremal values for $X_1$. \ \\
		In the case of equality, we then necessarily have that $|G|$ is a square and $|H|$ divides $\sqrt{|G|}$. 
		}
	\end{proof}

	For the sake of convenience, we give the bounds for bent functions $F\colon\F_p^n\to\F_p^m$ which will be considered in details in the following sections.
	
	\begin{theorem} \label{thm:CS}
		Let $F \colon \F_p^n \rightarrow \F_p^m$ be a bent function and $F^{-1}(\beta)$ the preimage set of $\beta \in \F_p^m$. Then for every $\beta \in \F_p^m$ the following inequality holds
		\begin{equation}\label{eq: preimage bound}
		    p^{n-m}-p^{n/2}+p^{n/2-m}\le |F^{-1}(\beta)| \le p^{n-m}+p^{n/2}-p^{n/2-m}.
		\end{equation}
		1. If $|F^{-1}(\alpha)|=p^{n-m}-p^{n/2}+p^{n/2-m}$ then $|F^{-1}(\beta)|=p^{n-m}+p^{n/2-m}$ for each $\beta \neq \alpha$. \\ 2. If $|F^{-1}(\alpha)|=p^{n-m}+p^{n/2}-p^{n/2-m}$ then $|F^{-1}(\beta)|=p^{n-m}-p^{n/2-m}$ for each $\beta \neq \alpha$. \ \\
		If equality takes place, then $m\le n/2$ and $n$ is even.
	\end{theorem}

	\begin{remark} \label{rem:boolean}
		\emph{1.} If we look for a moment at the Boolean case, i.e., $G=\F_2^n$ and $H=\F_2$, we see that the two extremal cases in Theorem~\ref{thm:CS} recover the well-known value distributions of Boolean bent functions. Indeed, for a Boolean bent function $f\colon\F_2^n\to\F_2$, we have
		\[W_f(0)=\sum_{x \in \F_2^n}{(-1)^{f(x)}}=|f^{-1}(0)|-|f^{-1}(1)| \in \{\pm 2^{n/2}\}.\]
		Since $|f^{-1}(0)|+|f^{-1}(1)|=2^n$, this implies $|f^{-1}(0)| = 2^{n-1} \pm 2^{n/2-1}$ which are exactly the two extremal cases in Theorem~\ref{thm:CS}. \ \\
		\noindent\emph{2.} For $p$ odd, $n$ even and $m=1$, the two extremal cases in Theorem~\ref{thm:CS} also recover the well-known value distributions of $p$-ary bent functions given in Theorem~\ref{th: Nyberg's p-ary distribution}. For $p$ odd, $n$ odd and $m=1$, compared to Theorem~\ref{th: Nyberg's p-ary distribution}, we obtain bounds on the cardinality of preimage sets of bent functions, which are not regular. \ \\
		\noindent\emph{3.} For the special case $p$ odd and $n$ even, the two extremal cases in Theorem~\ref{thm:CS} also cover the extremal distributions obtained in Equation~\eqref{eq:wang} found in \cite{CMP2021,Wang2023}. Moreover, these extremal cases not only recover the bound from~\cite[Corollary 2]{CMP2021}, but also show that the remaining elements in the image set are uniformly distributed between the remaining elements in $\F_{p}^n\setminus\{0\}$.
	\end{remark}
        Theorem~\ref{thm:CS general} can be used to identify a large class of surjective perfect nonlinear functions.
	\begin{corollary}
	\label{cor:p-ary}
		 Let $G$ and $H$ be two finite groups, and let  $F\colon G\to H$ be a perfect nonlinear function. If $|H| \le  \sqrt{|G|}$, then $F$ is surjective. A preimage set of size $1$ is only possible if $|H|=\sqrt{|G|}$. 
	\end{corollary}
	\begin{proof}
            We apply Theorem~\ref{thm:CS general}.
	    The function $F$ is surjective, if for all $\beta \in H$ we have \begin{equation}\label{eq: Bound for a preimage}
	        |F^{-1}(\beta)|\ge \frac{|G|}{|H|}-\sqrt{|G|}+\frac{\sqrt{|G|}}{|H|}\ge 1.
	    \end{equation}
	    From the latter inequality we have that$$ |G| + \sqrt{|G|} \ge |H|\cdot\left( \sqrt{|G|} +1  \right),$$
	    which is equivalent to $|H| \le \sqrt{|G|} $. A preimage set of size $1$ is only possible if $|H|=\sqrt{|G|}$.
	\end{proof}
	
	For perfect nonlinear functions beyond the ``square root bound'', the question about surjectivity becomes much more difficult to answer. In the following statement, we give a bound on the cardinality of the image set of a bent function, which has essentially been proven by Carlet~\cite{Carlet21} in a different context (namely, to give a  connection between nonlinearity and the cardinality of the image sets of mappings $F\colon\F_2^n\to\F_2^m$).
	\begin{proposition}\label{prop:imageset}
		 Let $G$ and $H$ be two finite groups, and let  $F \colon G \rightarrow H$ be a perfect nonlinear function. Then
		\begin{equation}\label{eq: Bound on the image set size}
		    |\image(F)| \geq \frac{|G|\cdot|H|}{|G|+|H|-1}.
		\end{equation}
	\end{proposition} 
	\begin{proof}
		We have from Equations~\eqref{eq:1} and~\eqref{eq:2}
		\begin{align*}
	\sum_{i=1}^{|\image(F)|} X_i^2&=|G|+\frac{|G|}{|H|}\cdot(|G|-1) \\
	\sum_{i=1}^{|\image(F)|} X_i&=|G|,
\end{align*}
	and again by the Cauchy-Schwarz inequality, the following holds
	\begin{equation}\label{eq: Images set cardinality bound aux}
	    |G|+\frac{|G|}{|H|}\cdot(|G|-1)=\sum_{i=1}^{|\image(F)|} X_i^2\geq\frac{|G|^2}{|\image(F)|}.
	\end{equation}
	The claim follows by solving Equation~\eqref{eq: Images set cardinality bound aux} for $|\image(F)|$.
	\end{proof}
	Now, we give the expression of the bound in Equation~\eqref{eq: Bound on the image set size} for bent functions $F\colon\F_p^n\to\F_p^m$ and  recover the well-known lower bound~\cite{KyureghyanP2008} for the planar case, i.e., $p$ is odd and $n=m$.
	\begin{corollary}\label{cor: Image sets+surhectivity}
	    Let $F\colon\F_p^n\to\F_p^m$ be a bent function. Then the following hold.
	    \begin{enumerate}
	        \item The cardinality of the image set of $F$ satisfies $$|\image(F)| \geq \frac{p^{2n}}{p^n+p^{n-m}(p^n-1)} > \frac{p^n}{1+p^{n-m}}.$$
	        \item If $p$ is odd and $n=m$, then $|\image(F)| \geq \frac{p^n+1}{2}$.
	        \item If $m\le n/2$, then $F$ is surjective.
	    \end{enumerate}
	\end{corollary}
	If the ``square root bound'' is violated, then the expression $\frac{|G|}{|H|}-\sqrt{|G|}+\frac{\sqrt{|G|}}{|H|}$ in Equation~\eqref{eq: Bound for a preimage} becomes negative. That means that our techniques cannot shed more light on the surjectivity of perfect nonlinear functions beyond the ``square root bound''. In the following example, we show that for small values of $p$, $n$ and $m\ge\lfloor n/2+1\rfloor$ surjective vectorial bent functions $F\colon\F_p^n\to\F_p^m$ exist.
	\begin{example}
	    Consider the planar function $F(x)=x^2$ on $\F_{p^n}$. Denote by $f_1,\ldots,f_n \colon\F_p^n\to\F_p$ the coordinate functions of $F$, i.e., $F(x)=(f_1(x),\ldots,f_n(x))$ for $x\in\F_{p^n}$. Let $F_{k}\colon\F_{p^n}\to\F_p^k$ be the vectorial bent function formed by the first $k$ coordinate functions of $F$, i.e., $F_{k}(x):=(f_1(x),\ldots,f_{k}(x))$. With Magma~\cite{Magma}, we checked that for the following values of $p$ and $n$ given in Table~\ref{table 1}, the functions $F_{k}\colon\F_{p^n}\to\F_p^k$, where $k=\lfloor n/2 \rfloor + 1$ are surjective.
	    \begin{table}[H]
	    \centering
	    \caption{Surjective vectorial bent functions $F_{\lfloor n/2 \rfloor + 1}\colon\F_{p^n}\to\F_p^{\lfloor n/2 \rfloor + 1}$}
            \label{table 1}
            \begin{tabular}{c|l}
            $p$ & $n$'s              \\ \hline
            3   & 5,6,7,8,9,10,11,12,13 \\
            5   & 5,6,7,8,9,10       \\
            7   & 5,6,7,8,9          \\
            11  & 5,6,7             
            \end{tabular}
        \end{table}
        More general, we expect that for a fixed $p$ and a sufficiently large $n$ there exist $m>\lfloor n/2 \rfloor + 1$, such that the functions $F_k\colon\F_{p^n}\to\F_p^k$ are surjective for all $\lfloor n/2 \rfloor + 1\le k\le m$, but not for $k\ge m+1$.  Consider $p=3,n=13$ and $m=10$. The functions $F_{k}\colon\F_{p^n}\to\F_p^{k}$ are surjective for all $1\le k\le6$ by Corollary~\ref{cor: Image sets+surhectivity}. With Magma~\cite{Magma}, we checked that for all $7\le k\le10$ the functions $F_{k}\colon\F_{p^n}\to\F_p^{k}$ are surjective as well. However, the functions $F_{k}\colon\F_{p^n}\to\F_p^{k}$ are not surjective for all $11\le k\le 13$.
	\end{example}
	Based on our observations on surjectivity of vectorial bent functions beyond the ``square root bound'', we formulate the following open problems for  vectorial bent functions $F\colon\F_p^n\to\F_p^m$, since most of the constructions are studied in this setting. Clearly, the asked questions do not lose their relevance for the case of perfect nonlinear functions beyond the ``square root bound'' on arbitrary groups.
    \begin{openproblem}
        Let $p$ be odd.
        \begin{enumerate}
            \item Find constructions of surjective vectorial bent functions $F\colon\F_p^n\to\F_p^m$ for $m\ge \lfloor n/2 \rfloor+1$.
            \item What is the maximum $m$ such that all vectorial bent functions $F\colon\F_p^n\to\F_p^m$ are surjective?
            \item What is the minimum $m$ such that all vectorial bent functions $F\colon\F_p^n\to\F_p^m$ are not surjective?
        \end{enumerate}
    \end{openproblem}
	
	\section{Almost balanced perfect nonlinear functions}\label{section: 3 Almost balanced vectorial bent functions}
    
	By Theorem~\ref{thm:CS general}, we know that if one preimage set $F^{-1}(\alpha)$ of a perfect nonlinear function $F\colon G\to H$ has either minimum or maximum cardinality, then the remaining elements in $H \setminus\{ \alpha \}$ are uniformly distributed between the remaining elements of $G \setminus F^{-1}(\alpha)$. This fact motivates the following definition.
	
	\begin{definition}
	     Let $G$ and $H$ be two finite groups. For a perfect nonlinear function $F \colon G \rightarrow H$, we call the first extremal value distribution in Theorem~\ref{thm:CS general} of \textit{type} ($-$), and the second extremal value distribution of \textit{type} ($+$).
      
        A perfect nonlinear function $F \colon G \rightarrow H$ is said to be \textit{almost balanced}, if its value distribution is extremal. Particularly, we say that $F$ is \textit{almost balanced of type} ($-$), if its value distribution is extremal of type ($-$), and \textit{almost balanced of type} ($+$), if its value distribution is extremal of type ($+$).

        For an almost balanced perfect nonlinear function, we say that $F^{-1}(\alpha)$ is the \textit{unique preimage} of $F$, if $\displaystyle |F^{-1}(\alpha)|=\frac{|G|}{|H|}\mp\sqrt{|G|}\pm\frac{\sqrt{|G|}}{|H|}$ (where the sign depends on the type).
	\end{definition}
 {
    Note that the two extremal distributions $(+)$ and $(-)$ belong to the solution $H_1=\pm \sqrt{|G|}-\frac{\sqrt{|G|}}{|H|}$, $H_i=\mp \frac{\sqrt{|G|}}{|H|}$ for all $i>1$ in Equations~\eqref{eq:3} and~\eqref{eq:4}. 
}

	From now on, we consider bent functions $F\colon\F_p^n\to\F_p^m$, since the problem of construction of bent functions is mostly considered in this setting. In the following subsections, we prove that many primary constructions of bent functions are, in fact, almost balanced. Moreover, we show how one can construct new almost balanced bent functions from known ones using secondary constructions.
	\subsection{Primary constructions}\label{subection: 3.1 Primary constructions}
    First, we consider three general classes of vectorial bent functions, namely the Maiorana-McFarland, the Desarguesian partial spread and the o-polynomial construction. We show that all these constructions yield almost balanced bent functions of the $(+)$ type.
    
	\begin{proposition} \label{prop:mm}
	    Let $n$ be even.
	    \begin{enumerate}
	        \item Let $F \colon \F_{p^{n/2}} \times \F_{p^{n/2}} \rightarrow \F_{p}^{m}$ be a Maiorana-McFarland bent function defined by
	        \begin{equation*}\label{eq: MM}
	            F(x,y)=L(x\pi(y))+\rho(y),
	        \end{equation*}
	        where $\pi\colon \F_{p^{n/2}} \rightarrow \F_{p^{n/2}}$ is a permutation, $\rho\colon\F_{p^{n/2}}\to\F_{p}^m$ is an arbitrary function, and $L\colon\F_{p^{n/2}}\to\F_{p}^m$ is a surjective linear mapping. Let $\pi(y^*)=0$ and $\alpha:=\rho(y^*)$ for $y^*\in\F_{p^m}$. Then $|F^{-1}(\alpha)|=p^{n-m}+p^{n/2}-p^{n/2-m}$ and $|F^{-1}(\beta)|=p^{n-m}-p^{n/2-m}$ for each $\beta \neq \alpha$, and hence $F$ is almost balanced of $(+)$ type.
	        \item Let $F \colon \F_{p^{n/2}} \times \F_{p^{n/2}} \rightarrow \F_{p}^{m}$ be a Desarguesian partial spread bent function defined by
	        \begin{equation*}\label{eq: PSap}
	            F(x,y)=\Psi(xy^{p^{n/2}-2}),
	        \end{equation*}
	        where $\Psi\colon\F_{p^{n/2}}\to\F_{p}^m$ is an arbitrary balanced function. Then $|F^{-1}(0)|=p^{n-m}+p^{n/2}-p^{n/2-m}$ and $|F^{-1}(\beta)|=p^{n-m}-p^{n/2-m}$ for each $\beta \neq 0$, and hence $F$ is almost balanced of $(+)$ type.
	    \end{enumerate}
	\end{proposition}
	\begin{proof}
	    In the both cases, it is enough to find an element $\alpha\in \F_{p}^m$, such that for the function $F \colon \F_{p^{n/2}} \times \F_{p^{n/2}} \rightarrow \F_{p}^{m}$ we have $|F^{-1}(\alpha)|=p^{n-m}+p^{n/2}-p^{n/2-m}$, since the uniformity of the other preimages follows immediately from Theorem~\ref{thm:CS}. \\ 
		\emph{1.} For $\alpha=\rho(y^*)$, the equation $L(x\pi(y))+\rho(y)=\alpha$ has $p^{n/2}$ solutions $(x,y^*)$, where $x\in\F_{p^{n/2}}$. Now let $y\neq y^*$ be fixed.  The equation $L(z)=\alpha-\rho(y)$ has $p^{n/2-m}$ solutions $z\in\F_{p^{n/2}}$, since $L$ is linear and surjective, and hence it is balanced, i.e., $|L^{-1}(\gamma)|=p^{n/2-m}$ for all $\gamma\in \F_{p^{m}}$. In turn, for every fixed $z\in\F_{p^{n/2}}$, the equation $z=x\pi(y)$ has a unique solution $x\in\F_{p^{n/2}}$ given by $x=z(\pi(y))^{-1}$. Hence, the equation $L(x\pi(y))+\rho(y)=\alpha$ has $p^{n/2-m}(p^{n/2}-1)$ additional solutions $(x,y)$, where $y\neq y^*$, and thus $p^{n-m}+p^{n/2}-p^{n/2-m}$ solutions in total. \\
		\emph{2.} Let $\alpha=\Psi(0)$. Consider the equation $\Psi(z)=\alpha$. Since $\Psi$ is balanced, we have $p^{n/2-m}$ solutions $z\in\F_{p^{n/2}}$. If $z\neq0$, then for a fixed $y\in\F_{p^{n/2}}^*$, the equation $z=xy^{p^{n/2}-2}=xy^{-1}$ has a unique solution $x=zy$, and hence the equation $\Psi(xy^{p^{n/2}-2})=\alpha$ has $(p^{n/2-m}-1)(p^{n/2}-1)$ solutions $(x,y)$, where $x=zy$ and $y\neq 0$. If $z=0$, then the set $\{ (x,0) \colon x\in \F_{p^n}\}\cup \{ (0,y) \colon y\in \F_{p^n}\}$ gives $p^{n/2+1}-1$ more solutions of the equation $\Psi(xy^{p^{n/2}-2})=\alpha$, and $p^{n-m}+p^{n/2}-p^{n/2-m}$ solutions in total.
    \end{proof}
    Recall the following definition of an o-polynomial~\cite{CarletMesnager2011}. For  an extensive summary of the known o-polynomials and in particular their relations to (hyper)ovals in finite geometry, we refer to~\cite{Mesnager2016}.
    \begin{definition}
        Let $k$ be any positive integer. A permutation polynomial $\Psi$ over $\F_{2^k}$ is called an \textit{o-polynomial (an oval polynomial)} if, for every $ a\in\F_{2^k}^*$, the function
        $$z\in\F_{2^k}\mapsto\begin{cases}
			\dfrac{\Psi(z+a)+\Psi(a)}{z}, & \mbox{if }z\neq 0\\
            0, & \mbox{if }z= 0
		 \end{cases} $$
		 is a permutation of $\F_{2^k}$.
    \end{definition}
    In the following statement, we show that bent functions obtained with the o-polynomial construction are also almost balanced of the $(+)$ type.
    \begin{proposition} \label{prop:opol}
        Let $n$ be even and $F \colon \F_{2^{n/2}} \times \F_{2^{n/2}} \rightarrow \F_{2^{n/2}}$ be an o-polynomial bent function defined by
	        \begin{equation}\label{eq: o-polynomial}
	            F(x,y)=x\Psi(yx^{2^{n/2-1}}),
	        \end{equation}
	        where $\Psi \colon \F_{2^{n/2}} \rightarrow \F_{2^{n/2}}$ is an o-polynomial. Then $|F^{-1}(0)|=2^{n/2+1}-1$ and $|F^{-1}(\beta)|=2^{n/2}-1$ for each $\beta \neq 0$, and hence $F$ is almost balanced of $(+)$ type.
    \end{proposition}
    \begin{proof}
        We have $F(x,y)=x\Psi(yx^{2^{n/2-1}})$, where $\Psi$ is an o-polynomial. Recall that as an o-polynomial $\Psi$ satisfies $\Psi(x)=0$ if and only if $x=0$. The equation $F(x,y)=0$ is thus only solvable if $xy=0$, so it has $2^{n/2+1}-1$ solutions. The uniformity of the other preimages follows then immediately from  Theorem~\ref{thm:CS}. 
    \end{proof}
    Now, we consider some monomial bent functions in even and odd characteristics. Again we find that these infinite families give almost balanced bent functions, but this time both $(+)$ and $(-)$ types occur. 
	\begin{proposition} \label{prop:monomials}
	Let $n$ be even.
	\begin{enumerate}
 		\item Let $n=2^{r+1}s$ with $r \geq 0$ and $s$ odd. Define $F \colon \F_{2^n} \rightarrow \F_{2^{n/2}}$ by $F(x)=\Tr^n_{n/2}(\lambda x^{2^{2^r}+1})$ where $\lambda$ is not a $(2^{2^r}+1)$-st power in $\F_{2^n}^*$. Then $|F^{-1}(0)|=1$ and $|F^{-1}(\beta)|=2^{n/2}+1$ for each $\beta \in \F_{2^{n/2}}^*$, and hence $F$ is almost balanced of type $(-)$.
		\item Let $n/2$ be odd and define $F \colon \F_{2^n} \rightarrow \F_{2^{n/2}}$ by $F(x)=\Tr^n_{n/2}(\lambda x^{4^i-2^i+1})$ where $\gcd(i,n)=1$ and $\lambda$ is a non-cube in $\F_{2^n}^*$. Then $|F^{-1}(0)|=1$ and $|F^{-1}(\beta)|=2^{n/2}+1$ for each $\beta \in \F_{2^{n/2}}^*$, and hence $F$ is almost balanced of type $(-)$.
		\item Let $p$ be odd and $F \colon \F_{p^n} \rightarrow \F_{p^{n/2}}$ be a vectorial bent function defined via $F(x)=\Tr^n_{n/2}(\lambda x^d)$ with $\gcd(d,p^{n/2}-1)=2$.  If $p^{n/2} \equiv 3 \pmod 4$ and $\lambda$ is a square, or $p^{n/2} \equiv 1 \pmod 4$ and $\lambda$ is a non-square, then $F$ is almost balanced of type $(+)$. In the other cases, $F$ is almost balanced of type $(-)$.
	\end{enumerate}
	\end{proposition}
	\begin{proof}
 		\noindent\emph{1.} These are well-known Gold vectorial bent functions (see, e.g.,~\cite[Theorem 6]{dong2013note}). We have $F(x)=0$ if and only if $\Tr^n_{n/2}(\lambda x^{2^{2^r}+1})=0$. Observe that $\Tr^n_{n/2}(x)=0$ if and only if $x \in \F_{2^{n/2}}$. We have $\gcd(2^{2^r}+1,2^{n/2}-1)=1$ since $(n/2)/\gcd(2^r,n/2)=1$  so $x \mapsto x^{2^{2^r}+1}$ is bijective on $\F_{2^{n/2}}$. In particular, each element in $\F_{2^{n/2}}$ is a $(2^{2^r}+1)$-st power. Now note that $\lambda x^{2^{2^r}+1}$ is never a $(2^{2^r}+1)$-st power, since $\lambda$ is a not a $(2^{2^r}+1)$-st power. So $\lambda x^{2^{2^r}+1} \in \F_{2^{n/2}}$ if and only if $x=0$. Thus $F(x)=0$ if and only if $x=0$.
  
		Now consider $F(x)=y$ with $y \neq 0$. We can write each $x \in \F_{2^{n/2}}^*$ uniquely as $x=ab$ where $a \in \F_{2^{n/2}}^*$, $b \in U_{2^m+1}=\{x \in \F_{2^n} \colon x^{2^{n/2}+1}=1\}$ since $(2^{n/2}-1)(2^{n/2}+1)=2^n-1$ and $\gcd(2^{n/2}-1,2^{n/2}+1)=1$. Then 
			\[F(ab)=a^{2^{2^r}+1}\Tr^n_{n/2}(\lambda b^{2^{2^r}+1})=y\]
			has for each $b\in U_{2^{n/2}+1}$ one unique solution $a$ (again since $x \mapsto x^{2^{2^r}+1}$ is bijective on $\F_{2^{n/2}}$). We conclude that each $y \neq 0$ has $2^{n/2}+1$ preimages. (The uniformity also follows immediately from Theorem~\ref{thm:CS}.)\\
		\noindent\emph{2.} These are the well-known Kasami vectorial bent functions (see, e.g.,~\cite[Theorem 7]{dong2013note}). We have $F(x)=0$ if and only if $\Tr^n_{n/2}(\lambda x^{4^i-2^i+1})=0$. Observe that $\Tr^n_{n/2}(x)=0$ if and only if $x \in \F_{2^{n/2}}$. Since ${n/2}$ is odd, $\gcd(4^i-2^i+1,2^{n/2}-1)=1$ (see, e.g.,~\cite[Lemma 3.8.]{kyureghyan2014inversion}) and $x \mapsto x^{4^i-2^i+1}$ is bijective on $\F_{2^{n/2}}$. In particular, each element in $\F_{2^{n/2}}$ is a $(4^i-2^i+1)$-th power. Now note that $\lambda x^{4^i-2^i+1}$ is never a $(4^i-2^i+1)$-th power, since $\gcd(4^i-2^i+1,2^n-1)=3$ (which can be readily checked) and $\lambda$ is a non-cube, so not a $(4^i-2^i+1)$-th power. So $\lambda x^{4^i-2^i+1} \in \F_{2^{n/2}}$ if and only if $x=0$. Thus $F(x)=0$ if and only if $x=0$. The uniformity follows immediately from Theorem~\ref{thm:CS}\\
	    \noindent\emph{3.}
		Without loss of generality, we can assume $d=2$. We then have $F(x)=0$ if and only if $\Tr^n_{n/2}(\lambda x^{2})=\lambda x^{2}+\lambda^{p^{n/2}} x^{2p^{n/2}}=0$. The non-zero roots $r$ must satisfy the equation $r^{2(p^{n/2}-1)}=-1/(\lambda^{p^{n/2}-1})$. If $-\lambda^{p^{n/2}-1}$ is a $2(p^{n/2}-1)$-st power, this has $2(p^{n/2}-1)$ non-zero solutions, meaning we have in total $2p^{n/2}-1$ solutions, so $F$ is of type ($+$). If $-\lambda^{p^{n/2}-1}$ is not a $2(p^{n/2}-1)$-st power, we conclude $F(x)=0$ if and only if $x=0$, i.e., we only have one solution and thus a type ($-$) function. The uniformity for the other preimage sizes follows from Theorem~\ref{thm:CS}. Note that $-\lambda^{p^{n/2}-1}$ is a $2(p^{n/2}-1)$-st power if and only if either $-1$ is a  $2(p^{n/2}-1)$-st power and $\lambda$ is a square or $-1$ is not a  $2(p^{n/2}-1)$-st power and $\lambda$ is a non-square. The result follows since  $-1$ is a  $2(p^{n/2}-1)$-st power if and only if $4(p^{n/2}-1)|p^n-1$ which is equivalent to $4|p^{n/2}+1$. 
	\end{proof}
	\begin{remark}
	    Note that all planar monomials $x \mapsto x^d$ on $\F_{p^n}$ satisfy $\gcd(d,p^n-1)=2$ (for a proof, see Corollary~\ref{cor:planar_monom} later) so Proposition~\ref{prop:monomials} in particular holds for the vectorial bent functions derived from all planar monomials. This means that Case 3 of Proposition~\ref{prop:monomials} yields almost balanced functions for all odd $p$ and all even $n$, since $x \mapsto x^2$ always yields a planar function (among other examples). 
	\end{remark}
	
	\subsection{Secondary constructions}\label{subection: 3.2 Secondary constructions}
    In this subsection, we show how one can construct almost balanced bent functions from the known ones. First, we consider the direct sum construction.
    \begin{definition}
        For two functions $F_1\colon\F_p^n\rightarrow\F_p^m$ and $F_2\colon\F_p^k\rightarrow\F_p^m$, the function $F\colon\F_p^n\times \F_p^k\rightarrow\F_p^m$ defined by $F(x,y):=F_1(x)+ F_2(y)$ is called the \textit{direct sum} of the functions $F_1$ and $F_2$.
    \end{definition}
    In the following statement, we give an expression of the cardinality of a preimage set of a direct sum $F(x,y)=F_1(x)+F_2(y)$ in terms of cardinalities of preimage sets of $F_1$ and $F_2$.
    \begin{proposition}
        Let $F_1\colon\F_{p}^{n}\to\F_p^{m},F_2\colon\F_{p}^{k}\to\F_p^{m}$ and $F\colon\F_{p}^{n}\times\F_{p}^{k}\to\F_p^{m}$ be defined as the direct sum of $F_1$ and $F_2$, i.e., $F(x,y)=F_1(x)+F_2(y)$ for $x\in\F_{p}^n$ and $y\in\F_{p}^k$. Then, for $c\in\F_{p}^m$ we have
        $$|F^{-1}(c)|=\sum_{a\in\F_{p}^m} |F_1^{-1}(a)|\cdot|F_2^{-1}(c-a)|. $$
    \end{proposition}
    \begin{proof}
        Let $c\in\F_{p}^m$. Clearly, $F(x,y)=c$ can be written as $F(x,y)=a+b$, where $a=F_1(x)$, $b=F_2(y)$ and $c=a+b$. In this way, the cardinality of the preimage set $F^{-1}(c)$ is given by
        \begin{equation}\label{eq: preimage of direct sum}
        |F^{-1}(c)|=\sum_{\substack{a,b\in\F_{p^n},\\a+b=c\in\F_p^m}} |F_1^{-1}(a)|\cdot|F_2^{-1}(b)|=\sum_{a\in\F_{p}^m} |F_1^{-1}(a)|\cdot|F_2^{-1}(c-a)|,
        \end{equation}
        completing the proof.
    \end{proof}
    Recall the following well-known result on the direct sum of two bent functions.
    \begin{proposition}\cite{ROTHAUS:1976}\label{prop: well-known}
        Let $F_1\colon\F_{p}^{n}\to\F_p^{m},F_2\colon\F_{p}^{k}\to\F_p^{m}$ be two  bent functions and let $F\colon\F_{p}^{n}\times\F_{p}^{k}\to\F_p^{m}$ be defined as  $F(x,y)=F_1(x)+F_2(y)$ for $x\in\F_{p}^n$ and $y\in\F_{p}^k$. Then $F$ is bent if and only if both $F_1$ and $F_2$ are bent.
    \end{proposition}
    In the following statement, we show that the direct sum of two almost balanced bent functions is almost balanced again.
\begin{proposition}\label{prop: direct sum}
    Let $F_1\colon\F_{p}^{n}\to\F_p^{m},F_2\colon\F_{p}^{k}\to\F_p^{m}$ be two almost balanced bent functions and let $F\colon\F_{p}^{n}\times\F_{p}^{k}\to\F_p^{m}$ be defined as  $F(x,y)=F_1(x)+F_2(y)$ for $x\in\F_{p}^n$ and $y\in\F_{p}^k$. Then the  following hold.
    \begin{enumerate}
        \item If $F_1$ and $F_2$ are both of the $(+)$ type, then the direct sum $F$ is of the $(+)$ type as well.
        \item If $F_1$ and $F_2$ are both of the $(-)$ type, then the direct sum $F$ is of the $(+)$ type.
        \item If $F_1$ is of the $(+)$ type, and $F_2$ is of the $(-)$ type, then the direct sum $F$ is of the $(-)$ type.
    \end{enumerate}
\end{proposition}    
\begin{proof}
    Let $F_1^{-1}(a_1)$ and $F_2^{-1}(a_2)$ be the unique preimages of $F_1$ and $F_2$, respectively. Define $c^*\in\F_p^m$ as $c^*=a_1+a_2$. Since $F$ is bent by Proposition~\ref{prop: well-known}, it is enough to show that $|F^{-1}(c^*)|=p^{n+k-m}- p^{(n+k)/2}+ p^{(n+k)/2-m}$, i.e., $F$ is almost balanced of ($-$) type, or $|F^{-1}(c^*)|=p^{n+k-m}+ p^{(n+k)/2}- p^{(n+k)/2-m}$, i.e., $F$ is almost balanced of ($+$) type,  since by Theorem~\ref{thm:CS} the uniformity of the other preimages is forced automatically. From Equation~\ref{eq: preimage of direct sum}, we have
    \begin{equation} \label{eq: preimage size of direct sum general}
        \begin{split}
        |F^{-1}(c^*)|=&\sum_{a\in\F_{p}^m} |F_1^{-1}(a)|\cdot|F_2^{-1}(c^*-a)|\\
        =&|F_1^{-1}(a_1)|\cdot|F_2^{-1}(a_2)|+\sum_{a\in\F_{p}^m\setminus\{a_1\}} |F_1^{-1}(a)|\cdot|F_2^{-1}(c^*-a)|\\
        =&|F_1^{-1}(a_1)|\cdot|F_2^{-1}(a_2)|+(p^m-1)\cdot |F_1^{-1}(a)|\cdot|F_2^{-1}(c^*-a)|,
        \end{split}
    \end{equation}
    where $a\in\F_{p}^m\setminus\{a_1\}$. \\
    \emph{1.} Since $F_1$ and $F_2$ are both of the ($+$) type, we get from Equation~\eqref{eq: preimage size of direct sum general} that the cardinality of $F^{-1}(c^*)$ is given by
     \begin{equation*} 
        \begin{split}
        |F^{-1}(c^*)|=&\left(-p^{\frac{k}{2}-m}+p^{k-m}+p^\frac{k}{2}\right) \cdot\left(-p^{\frac{n}{2}-m}+p^{n-m}+p^\frac{n}{2}\right)\\+&\left(p^m-1\right)\cdot \left(p^{k-m}-p^{\frac{k}{2}-m}\right) \cdot\left(p^{n-m}-p^{\frac{n}{2}-m}\right)\\
        =&p^{n+k-m}+ p^\frac{n+k}{2}- p^{\frac{n+k}{2}-m},
        \end{split}
    \end{equation*} 
    from what follows that $F$ is almost balanced of the ($+$) type.\\
    \emph{2.}  Since $F_1$ and $F_2$ are both of the ($-$) type, we get from Equation~\eqref{eq: preimage size of direct sum general} that the cardinality of $F^{-1}(c^*)$ is given by
    \begin{equation*} 
        \begin{split}
        |F^{-1}(c^*)|=&\left(p^{\frac{k}{2}-m}+p^{k-m}-p^\frac{k}{2}\right) \cdot\left(p^{\frac{n}{2}-m}+p^{n-m}-p^\frac{n}{2}\right)\\+&\left(p^m-1\right)\cdot \left(p^{\frac{k}{2}-m}+p^{k-m}\right) \cdot \left(p^{\frac{n}{2}-m}+p^{n-m}\right)\\
        =&p^{n+k-m}+ p^\frac{n+k}{2}- p^{\frac{n+k}{2}-m},
        \end{split}
        \end{equation*}
        from what follows that $F$ is almost balanced of the ($+$) type.\\
    \emph{3.} Since $F_1$ is of the ($+$) type, and $F_2$ is of the ($-$) type, we get from Equation~\eqref{eq: preimage size of direct sum general} that the cardinality of $F^{-1}(c^*)$ is given by
    \begin{equation*} 
        \begin{split}
        |F^{-1}(c^*)|=&\left(-p^{\frac{k}{2}-m}+p^{k-m}+p^\frac{k}{2}\right)\cdot \left(p^{\frac{n}{2}-m}+p^{n-m}-p^\frac{n}{2}\right)\\+&\left(p^m-1\right)\cdot \left(p^{k-m}-p^{\frac{k}{2}-m}\right)\cdot \left(p^{\frac{n}{2}-m}+p^{n-m}\right)\\
        =&p^{n+k-m}- p^\frac{n+k}{2}+ p^{\frac{n+k}{2}-m},
        \end{split}
    \end{equation*}
    from what follows that $F$ is almost balanced of the ($-$) type.
\end{proof}
Finally, we show that all possible bent functions which are ``contained'' in a given almost bent function are almost balanced of the same type. 
\begin{proposition}\label{prop: composition with surjective}
    Let $F\colon\F_p^n\to\F_p^m$ be an almost balanced surjective bent function and let $L\colon\F_p^m\to\F_p^{k}$ be a surjective linear mapping. Then the following hold.
    \begin{enumerate}
        \item If $F$ is of the $(+)$ type, then $L\circ F\colon\F_p^n\to\F_p^{k}$ is of the $(+)$ type as well.
        \item If $F$ is of the $(-)$ type, then $L\circ F\colon\F_p^n\to\F_p^{k}$ is of the $(-)$ type as well.
    \end{enumerate}
    Consequently, if $F$ is of the $(+)$ type (resp. $(-)$ type), then every component bent function is of the $(+)$ type (resp. $(-)$ type).
\end{proposition}
\begin{proof}
For $b\in\F_p^m$, let $F^{-1}(b)$ be the unique preimage of $F$. Denote by $c=L(b)\in\F_p^k$. Since every non-zero component function of $F$ is bent, we have that $L\circ F$ is bent as well, and hence it is enough to show that $|(L\circ F)^{-1}(c)|=p^{n-k}\pm p^{n/2}\mp p^{n/2-k}$, since by Theorem~\ref{thm:CS} the uniformity of the other preimages is forced automatically. Since $L\colon\F_p^m\to\F_p^{k}$ is a surjective linear mapping, it is balanced, and hence $|L^{-1}(a)|=p^{m-k}$ for all $a\in\F_{p}^{k}$. In this way, the cardinality of $(L\circ F)^{-1}(c)$ is given by
\begin{equation}\label{eq: preimage set composition}
    |(L\circ F)^{-1}(c)|=|F^{-1}(b)|+(p^{m-k}-1)|F^{-1}(a)|,
\end{equation}
where $a\in\F_{p}^m\setminus\{b\}$.\\
\emph{1.} If $F$ is of the ($+$) type, then $|F^{-1}(b)|=p^{n-m}+p^{n/2}-p^{\frac{n}{2}-m}$ and $|F^{-1}(a)|=p^{n-m}-p^{\frac{n}{2}-m}$ for $a\in\F_{p}^m\setminus\{b\}$. Then, from Equation~\eqref{eq: preimage set composition}, we have $|(L\circ F)^{-1}(c)|=p^{\frac{1}{2} (n-2 k)} \left(p^k+p^{n/2}-1\right)$.\\
\emph{2.} If $F$ is of the ($-$) type, then $|F^{-1}(b)|=p^{n-m}-p^{n/2}+p^{\frac{n}{2}-m}$ and $|F^{-1}(a)|=p^{n-m}+p^{\frac{n}{2}-m}$ for $a\in\F_{p}^m\setminus\{b\}$. Then, from Equation~\eqref{eq: preimage set composition}, we have $|(L\circ F)^{-1}(c)|=p^{\frac{1}{2} (n-2 k)} \left(-p^k+p^{n/2}+1\right)$.

The last claim follows by considering the composition  $L\circ F$, where $L\colon\F_p^m\to\F_p$ is linear and surjective.
\end{proof}

In particular, these results show that almost balanced bent functions of both types exist for all possible choices of $p,n,m$ (of course, the trivial restrictions $n$ even and $m\leq n/2$ follow immediately from the definition of the two types). 

\begin{theorem}\label{th:existence}
    Almost balanced bent functions $F\colon\F_p^n\to\F_p^m$ of types $(+)$ and $(-)$ exist for all $m\le n/2$, where $n\in \N$ is an arbitrary even number and $p$ is an arbitrary prime number.
\end{theorem}
\begin{proof}
   Follows from the application of Proposition~\ref{prop: composition with surjective} to almost balanced vectorial bent functions $F\colon\F_p^n\to\F_p^{n/2}$ from primary constructions given in Propositions~\ref{prop:mm}, \ref{prop:opol} and \ref{prop:monomials}, all of which are surjective. 
\end{proof}

\section{Value distributions and the Walsh transform}\label{section: 4 Preimage sets and the Walsh transform}
In this section, we develop the connection between the Walsh transform of a bent function and its value distribution.	Particularly, we show that with the knowledge of the Walsh transform, one can get precise information about the value distribution, and vice versa.
 

Recall the following well-known result, which will play an important role in the connection between value distributions and the Walsh transform. 
\begin{lemma}\cite{Nyberg91}\label{lem:gauss}
    Let $p$ be an odd prime. Then
    \[\sum_{r=1}^{p-1} a_r \zeta_p^r=\begin{cases}
    \sqrt{p}, & \text{if $p \equiv 1 \pmod 4$} \\
    i\sqrt{p}, & \text{if $p \equiv 3 \pmod 4$}
    \end{cases}\]
    has a (unique) integer solution $a_r=\left(\frac{r}{p}\right)$ for all $r$, where $\left(\frac{r}{p}\right) \in \{-1,1\}$ is the Legendre symbol.
\end{lemma}
As we have seen previously, many of the known constructions yield almost balanced bent functions. Interestingly, these preimage set distributions actually force a plateaued function to be bent. Since plateaued functions are much more prevalent than bent functions, this again underscores the special nature of the almost balanced bent functions we introduced.

\begin{theorem} \label{thm:walsh}
	Let $F \colon \F_{p^n} \rightarrow \F_{p^m}$ be a plateaued function with the preimage distribution of type $(+)$ or $(-)$ where the unique preimage is $F^{-1}(0)$. Then $F$ is bent. More precisely, we have $W_F(b,0)=-p^{n/2}$ for all $b \in \F_{p^m}^*$ for the type $(-)$ function and $W_F(b,0)=p^{n/2}$ for all $b \in \F_{p^m}^*$ for the type $(+)$ function.
\end{theorem}
\begin{proof}
 Let us first consider the type $(-)$ distribution, so $|F^{-1}(0)|=p^{n-m}-p^{n/2}+p^{n/2-m}$ and $|F^{-1}(y)|=p^{n-m}+p^{n/2-m}$ for non-zero $y$. 
Then $|\{x \notin F^{-1}(0) \colon \Tr(bF(x))=c\}|$ is divisible by $p^{n-m}+p^{n/2-m}$ for any $c \in \F_p$ and any $b \in \F_{p^m}^*$. In particular,
\begin{equation*}
	W_F(b,0) \equiv p^{n-m}-p^{n/2}+p^{n/2-m} \equiv -p^{n/2} \pmod {p^{n-m}+p^{n/2-m}}. 
	\end{equation*}
			Since $F$ is plateaued, $W_F(b,0)$ can only attain the values with absolute value $p^{(n+k)/2}$ with $k\geq 0$ or $0$ for $b \neq 0$. Clearly, $W_F(b,0)=0$ is not possible. 
			
			By the congruence above, $p^{n/2}+W_F(b,0) = C (p^{n-m}+p^{n/2-m})$ for some $C \in \Z[\zeta_p]$. This simplifies to
			\begin{equation} \label{eq:walsh}
			    	C(p^{n/2}+1)=p^m(1 + W_F(b,0)/p^{n/2}).
			\end{equation}

			Let us first focus on the case $p=2$, so $C \in \Z$. Then $W_F(b,0)=\pm 2^{n/2+k}$ with $k\leq n/2$. Observe that $2^{n/2}+1$ can divide $2^k \pm 1$ only if $k \in \{0,n/2\}$. If $k=n/2$ then $\Tr(bF(x))$ is constant $0$, in particular, the image sets of $bF(x)$ would be contained in a hyperplane of size $2^{n-1}$. This contradicts Proposition~\ref{prop:imageset}. We conclude that $k=0$ and $W_F(b,0)=-2^{n/2}$ for all $b \neq 0$. Since $F$ is plateaued, it is thus bent.
			
			Now consider $p>2$. Then $C=\sum_{r=1}^{p-1}a_r \zeta_p^r$ with integer coefficients $a_r$. Further, recall that $\sum_{r=1}^{p-1}\zeta_p^r=-1$, so Equation~\eqref{eq:walsh} becomes
			\begin{equation} \label{eq:walsh2}
		    	\sum_{r=1}^{p-1}(p^m+(p^{n/2}+1)a_r)\zeta_p^r=p^{m-n/2}W_F(b,0)
			\end{equation}
			By \cite[Theorem 2]{Hyun16}, we have $W_F(b,0)= \epsilon p^{\frac{n+k}{2}} \zeta_p^t$ for some $0\leq k \leq n$ and $t$ where $\epsilon \in \{1,-1\}$ if $n+k$ is even or $n+k$ is odd and $p \equiv 1 \pmod 4$ and $\epsilon \in \{i,-i\}$ if $n+k$ is odd and $p \equiv 3 \pmod 4$. Let us first deal with the case that $W_F(b,0)= \pm p^{\frac{n+k}{2}} \in \Z$. Equation~\eqref{eq:walsh2} then states
			\[\sum_{r=1}^{p-1}(p^m+(p^{n/2}+1)a_r\mp p^{m+k/2})\zeta_p^r=0.\]
			We conclude that $p^m+(p^{n/2}+1)a_r\mp p^{m+k/2}=0$ for all $r$, that is, 
			\[p^m\frac{\pm p^{k/2}-1}{p^{n/2}+1}=a_r \in \Z.\]
			We can now argue as in the $p=2$ case that, for divisibility reasons, we must have the minus sign in the equation and $k=0$, leading to a bent function and $W_F(b,0)=-p^{n/2}$.
			
			It remains to exclude the case $W_F(b,0) \notin \Z$. Let us first deal with the case $\epsilon=\pm 1$. Then, again from Equation~\eqref{eq:walsh2}, we have that
			\[\sum_{r=1}^{p-1}(p^m+(p^{n/2}+1)a_r)\zeta_p^r-p^{\frac{n+k}{2}}\epsilon \zeta_p^t=0\]
			for some $1 \leq t \leq p-1$,
			leading to $p^m+(p^{n/2}+1)a_r=0$ for any $r \neq t$. This is clearly never satisfied for $a_r \in \Z$, so this case cannot occur. Let us now assume $\epsilon = \pm i$. Then 
			\begin{equation}\label{eq:walsh3}
			    \sum_{r=1}^{p-1}(p^m+(p^{n/2}+1)a_r)\zeta_p^r=\pm p^{\frac{n+k-1}{2}} \left(\sum_{r=1}^{p-1}c_r \zeta_p^{r+t}\right),
			\end{equation}
			using Lemma~\ref{lem:gauss}, where $c_r=\left(\frac{r}{p}\right)$. If $t=0$ this means $p^m+(p^{n/2}+1)a_r\mp p^{\frac{n+k-1}{2}}c_r=0$ for all $r$. This is equivalent to $a_r=p^mc_r \frac{\pm p^{\frac{n+k-m-1}{2}}-c_r}{p^{n/2}+1}$. Since both $c_r=-1,1$ occur, this cannot always be an integer for a fixed $k$, yielding a contradiction. If $t \neq 0$, then $\zeta_p^n$ occurs on the right hand side of Equation~\eqref{eq:walsh3}, yielding $p^m+(p^{n/2}+1)a_r\mp p^{\frac{n+k-1}{2}}c_r=\pm 1$ for all but one $r$. Since both $c_r=-1,1$ occur, $2p^{\frac{n+k-1}{2}}$ must be divisible by $p^{n/2}+1$, which is clearly not the case. We get again a contradiction. This concludes the $(-)$ case.
			
			The second extremal case $|F^{-1}(0)|=p^{n-m}+p^{n/2}-p^{n/2-m}$ and $|F^{-1}(y)|=p^{n-m}-p^{n/2-m}$ for each $y \neq 0$ can be dealt with in a similar fashion. This time, we have 
			\[W_F(b,0) \equiv p^{n/2} \pmod{p^{n-m}-p^{n/2-m}}\]
			and with the same argumentation as above with only a change of signs throughout, $F$ is bent where $W_F(b,0)=p^{n/2}$ for all $b \neq 0$.
\end{proof}

\begin{remark}
    Note that the condition in Theorem~\ref{thm:walsh} that the unique preimage is the preimage of 0 is not restrictive. Indeed, one can shift a plateaued function always to achieve this without changing the preimage set sizes or losing the plateaued property. Of course, such a shift will however change the signs of the Walsh transform.
\end{remark}

The following result is a direct consequence of Theorem~\ref{thm:walsh}, which gives a purely combinatorial way (via preimage set sizes) to check for the regularity of bent functions.
\begin{corollary} \label{cor:regular}
    Let $p$ be odd and $n$ be even. Let $F\colon\F_p^n\to\F_p^m$ be a bent function.
    \begin{enumerate}
        \item If $F$ is of type $(+)$  it is inequivalent to a function of the $(-)$ type.
        \item If $F$ is of the $(+)$ type (or equivalent to a $(+)$ type function) then each weakly regular component function of $F$ is regular.
        \item If $F$ is of the $(-)$ type (or equivalent to a $(-)$ type function) then each weakly regular component function of $F$ is not regular.
    \end{enumerate}
\end{corollary}
\begin{proof}
    This follows from Theorem~\ref{thm:walsh} and the fact that for $p$ odd and $n$ even regularity is preserved under equivalence, see~\cite[p. 233]{CM2013}
\end{proof}

Theorem~\ref{thm:walsh} also shows that almost balanced bent functions of type $(+)$ or $(-)$ have a very special Walsh transform in the sense that (potentially after a shift) $W_F(b,0)$ is always plus or minus $p^{n/2}$. Interestingly, this is a precise characterization of these distributions, i.e., the converse also holds:

\begin{proposition} 
    Let $F \colon \F_{p^n} \rightarrow \F_{p^m}$ be a bent function such that $W_F(b,0)=p^{n/2}$ (resp. $-p^{n/2}$) for all $b \in \F_{p^m}^*$. Then $F$ is almost balanced of type $(+)$ (resp. $(-)$) and the unique preimage is $F^{-1}(0)$.
\end{proposition}
\begin{proof}
    We only deal with the $(+)$ case, the $(-)$ case works identically up to changing the sign throughout. We have
    \begin{equation*} 
        \sum_{b \in \F_{p^m}}W_F(b,0)=\sum_{x \in \F_{p^n}}\sum_{b \in \F_{p^m}} \zeta_p^{\Tr(bF(x))}=p^m\cdot |\{x \in \F_{p^n} \colon F(x)=0\}|.
    \end{equation*}
Counting another way, we also have
    \begin{equation*} 
        \sum_{b \in \F_{p^m}}W_F(b,0)=p^n+\sum_{b \in \F_{p^m}^*}W_F(b,0)=p^n+p^{n/2}(p^m-1).
    \end{equation*}
    Comparing these two equations yields 
    \[|\{x \in \F_{p^n} \colon F(x)=0\}|=\frac{p^n+p^{n/2}(p^m-1)}{p^m}=p^{n-m}+p^{n/2}-p^{n/2-m}.\]
    By Theorem~\ref{thm:CS}, it follows that $F$ is necessarily of type $(+)$.
\end{proof}

For bent functions, where every component function is weakly regular with the same sign (for instance, if all component functions are regular), the possible preimage set sizes are actually very limited. Note that this in particular holds for all Boolean bent functions, where we consider all component functions to be regular by default.

\begin{theorem} \label{thm:regular_constraints}
    Let $F \colon \F_{p^n} \rightarrow \F_{p^m}$ be a bent function where $W_F(b,0)=\epsilon p^{n/2}\zeta_p^{r_b}$ with $\epsilon \in \{\pm1,\pm i\}$ for all $b \in \F_{p^m}^*$. Then $n$ is even, $m\leq n/2$ and for any $a \in \F_{p^m}$ we have
    \begin{itemize}
        \item If $\epsilon \in \{1,i\}$:     \[|F^{-1}(a)|=p^{n-m}+p^{n/2}-p^{n/2-m}(pk_a+1),\]
    where $0 \leq k_a \leq \frac{p^m-1}{p-1}$. In particular,  $|F^{-1}(a)| \geq p^{n-m}-\frac{p^m-1}{p-1}p^{n/2-m}$.
        \item If $\epsilon \in \{-1,-i\}$:  :  
        \[|F^{-1}(a)|=p^{n-m}-p^{n/2}+p^{n/2-m}(pk_a+1),\]
    where $0 \leq k_a \leq \frac{p^m-1}{p-1}$. In particular,  $|F^{-1}(a)| \leq p^{n-m}+\frac{p^m-1}{p-1}p^{n/2-m}$.
    \end{itemize}
    Additionally, $k_0=|\{b \in \F_{p^m}^* \colon W_F(b,0)=p^{n/2}\epsilon \zeta_p\}|$ .
\end{theorem}
\begin{proof}
    We start with the $\epsilon\in \{1,i\}$ case.
    We have
    \begin{equation*} 
        \sum_{b \in \F_{p^m}}W_F(b,0)=\sum_{x \in \F_{p^n}}\sum_{b \in \F_{p^m}} \zeta_p^{\Tr(bF(x))}=p^m\cdot |\{x \in \F_{p^n} \colon F(x)=0\}|.
    \end{equation*}
Counting another way, we also have
    \begin{equation} \label{eq:compare_1}
        \sum_{b \in \F_{p^m}}W_F(b,0)=p^n+\sum_{b \in \F_{p^m}^*}W_F(b,0)=p^n+p^{n/2}\left(\sum_{r=0}^{p-1}\epsilon l_r\zeta_p^r\right),
    \end{equation}
    where $l_r=|\{b \in \F_{p^m}^* \colon F_b^*(0)=r\}|$.
    Comparing these two equations yields 
    \begin{equation} \label{eq:compare}
        |\{x \in \F_{p^n} \colon F(x)=0\}|=p^{n-m}+p^{n/2-m}\left(\sum_{r=0}^{p-1}\epsilon l_r\zeta_p^r\right).
    \end{equation}
    Observe that the left-hand side of Equation~\eqref{eq:compare} is an integer, so the right-hand side also has to be an integer. We divide the proof now into two cases: \\
    \textbf{$n$ even:} Now $\epsilon=1$ and $p^{n/2-m}$ is rational, implying that $\sum_{r=1}^{p-1}l_r\zeta_p^r$ has to be an integer, which in turn implies that all $l_r=l_1=:l$, for all $r>0$ and $\sum_{r=1}^{p-1}l_r\zeta_p^r=-l$. Clearly, $\sum_{r=0}^{p-1}l_r=(p-1)l+l_0=p^m-1$, leading to $l_0=p^m-1-(p-1)l$. Substituting this into Equation~\eqref{eq:compare} yields
    \[|\{x \in \F_{p^n} \colon F(x)=0\}|=p^{n-m}+p^{n/2-m}(p^m-1-pl)=p^{n-m}+p^{n/2}-p^{n/2-m}(pl+1)\]
    as claimed where $0\leq l \leq \frac{p^m-1}{p-1}$. Note that $pl+1$ is not divisible by $p$, so $p^{n/2-m}$ has to be an integer, leading to $m\leq n/2$.\\
    \textbf{$n$ odd:} Now  $p^{n/2-m}$ is not rational, so Equation~\eqref{eq:compare} implies that $\sum_{r=0}^{p-1}l_r\zeta_p^r$ has to be $0$ or $\sqrt{p}$ times an integer. In the first case, all $l_r$ have to be the same, contradicting $p^m-1=\sum_{r=0}^{p-1}l_r$. In the other case, we have necessarily by Lemma~\ref{lem:gauss} that $l_r=l \cdot \left(\frac{r}{p}\right)+l_0$ for all $r>0$ and some $l$. Then
    \[p^m-1=\sum_{r=0}^{p-1}l_r=l_0+\sum_{r=1}^{p-1}l\left(\frac{r}{p}\right)+(p-1)l_0=pl_0.\]
    This is a contradiction since $p$ does not divide the left-hand side.
    
    Clearly, shifting the function preserves the bentness as well as the sizes of the preimage sets, so we get the result not only for the preimage of $0$ but for all preimages.
 
    For $\epsilon\in \{-1,-i\}$, we get on the right-hand side of Equation~\eqref{eq:compare_1} $p^n-p^{n/2}\left(\sum_{r=0}^{p-1}\epsilon l_r\zeta_p^r\right)$, i.e., just a change of signs. Then, the same argument as for the regular case leads to the result. 
    
    The minimal image set size in the regular case is clearly reached by setting $l=\frac{p^m-1}{p-1}$ and substituting this into the equation yields 
    \begin{align*}
          |F^{-1}(a)|&=p^{n-m}+p^{n/2}-p^{n/2-m}\left(\frac{p}{p-1}(p^m-1)+1\right)\\
          &=p^{n-m}-\frac{1}{p-1}\left(p^{n/2}-p^{n/2-m}\right)=p^{n-m}-\frac{p^m-1}{p-1}p^{n/2-m}.  
    \end{align*}
    Similarly, the maximal image set size for $\epsilon\in \{-1,-i\}$ case is reached by setting $l=\frac{p^m-1}{p-1}$ and the result follows again immediately.
\end{proof}    

\begin{remark}
    The condition that $W_F(b,0)=\epsilon p^{n/2}\zeta_p^{r_b}$ with $\epsilon \in \{\pm1,\pm i\}$ for all $b \in \F_{p^m}^*$ is less restrictive than it might appear. It holds in particular for vectorial bent function where all component functions are regular, and for all Boolean bent functions.
\end{remark}

Note that Theorem~\ref{thm:regular_constraints} in particular gives a simple proof of Nyberg's bound, i.e., it shows that $m\leq n/2$ for Boolean bent functions. In fact, the result is a stronger version of Nyberg's original result~\cite[Theorem 3.2.]{KaiNy91} which showed that all preimage set sizes of vectorial bent functions that have only regular component functions are of the form $p^{n-m/2} \cdot k$ where $k$ is not divisible by $p$. Theorem~\ref{thm:regular_constraints} gives both more precise information on the preimage set sizes as well as generalizes the result to a wider set of bent functions.

For $p=2$, the constraints on the Walsh transform in Theorem~\ref{thm:regular_constraints} are trivial and can be dropped. In this case, the possible values coincide and the bounds coincide with the ones from Theorem~\ref{thm:CS} (while in the $p$-ary case the bounds from Theorem~\ref{thm:regular_constraints} are better). In the Boolean case we can in fact derive an extra condition.
\begin{theorem} \label{thm:boolean_constraints}
        Let $F \colon \F_{2^n} \rightarrow \F_{2^m}$ be a Boolean bent function. Then for any $a \in \F_{2^m}$ we have 
    \[|F^{-1}(a)|=2^{n-m}+2^{n/2}-2^{n/2-m}(2k_a+1),\]
    where $k_a=|\{b \in \F_{2^m}^* \colon W_{F}(b,0)=-2^{n/2}\cdot (-1)^{\Tr(ab)}\}|$ and all $k_a$ have the same parity. 
\end{theorem}
\begin{proof}
    The result on the preimage set sizes follows immediately from Theorem~\ref{thm:regular_constraints}. It remains to show that all $k_a$ have the same parity. We have $k_0=|\{b \in \F_{2^m}^* \colon W_F(b,0)=-2^{n/2}\}|$. 
    Further
    \[ W_{F+a}(b,0)=(-1)^{\Tr(ab)}W_F(b,0).\]
    In particular, we see that $W_{F+a}(b,0)$ coincides with $W_{F}(b,0)$ if $\Tr(ab)=0$ and does not coincide if $\Tr(ab)=1$. So we have $2^{m-1}$ sign changes. In particular, the number of $+$ signs that get turned into $-$ signs has the same parity as the $-$ signs that get turned into $+$ signs. Consequently, $k_0=|\{b \in \F_{2^m}^* \colon W_F(b,0)=-2^{n/2}\}|$ and $s_a=|\{b \in \F_{2^m}^* \colon W_{F+a}(b,0)=-2^{n/2}\}|$ have the same parity for any $a$. But for $a \in \F_{2^m}$ we have again
     \begin{equation*} 
        \sum_{b \in \F_{2^m}}W_{F+a}(b,0)=\sum_{x \in \F_{2^n}}\sum_{b \in \F_{2^m}} (-1)^{\Tr(b(F(x)+a))}=2^m\cdot |F^{-1}(a)|.
    \end{equation*}
    and 
    \begin{equation*} 
        \sum_{b \in \F_{2^m}}W_{F+a}(b,0)=2^n+\sum_{b \in \F_{2^m}^*}W_{F+a}(b,0)=2^n+2^{n/2}\left(2^m-1-2s_a\right).
    \end{equation*}
    By comparison, we see that $s_a=k_a$, so all $k_a$ have the same parity as claimed.
\end{proof}

In the case that $n$ is odd (which necessarily implies that $p$ is odd) we also get more precise information on the possible size of the preimages. Note that here we do not need any additional conditions on the Walsh transform.
\begin{theorem} \label{thm:m odd}
    Let $F\colon \F_{p^n} \rightarrow \F_{p^m}$ be a bent function with $p,n$ odd. Then for any $a\in\F_{p^m}$ we have $|F^{-1}(a)|=p^{n-m}$ or
        \[|F^{-1}(a)|=p^{n-m}\pm p^{(n+1)/2-m} \sum_{r=1}^{p-1} \left(k\left(\frac{r}{p}\right)+k_0\right) ,\]
        where $1 \leq k \leq \frac{p^m-1}{p-1}-(p-1)k_0$ and $k_0$ is a non-negative integer.
\end{theorem}
\begin{proof}
    We get again (like Equation~\eqref{eq:compare} just without assuming additional conditions on the Walsh transform)
        \begin{equation*} 
        |F^{-1}(0)|=p^{n-m}+p^{n/2-m}\left(\sum_{r=0}^{p-1}\epsilon_r \delta k_r\zeta_p^r\right),
    \end{equation*}
    where $k_r$ are non-negative integers satisfying $\sum_r k_r\leq p^m-1$,  $\epsilon_r \in \{1,-1\}$ and $\delta \in \{1,i\}$ depending on $p$. 
    Since $n$ is odd, we know that $p^{n/2-m}=\sqrt{p}\cdot p^{(n-1)/2-m}$ is not rational. Then, either all $k_r=0$ (leading to $|F^{-1}(0)|=p^{n-m}$) or, using Lemma~\ref{lem:gauss}, we have $\epsilon_rk_r=\epsilon_1k\left(\frac{r}{p}\right)+k_0$ for all $r>0$ with $1\leq k \leq \frac{p^m-1}{p-1}-(p-1)k_0$, leading to
           \begin{equation*} 
        |\{x \in \F_{p^n} \colon F(x)=0\}|=p^{n-m}\pm p^{(n+1)/2-m} \sum_{r=1}^{p-1} \left(k\left(\frac{r}{p}\right)+k_0\right).
    \end{equation*}
    Again, shifting does not affect the preimage set sizes, so we get the same conditions also on the preimages of non-zero elements. 
\end{proof}

\section[Value distributions of bent functions with a small output space]{Value distributions of bent functions $F\colon\mathbb{F}_p^n \rightarrow  \mathbb{F}_p^m$ with  small values of $m$}\label{section: 5 Boolean (n,2)-bent functions are almost balanced}
While listing all possible preimage size distributions for vectorial bent functions in complete generality seems to be a very difficult and out-of-reach task, for small values of $m$ such results are possible to obtain from our work in the previous sections. We want to remind the reader that the case $m=1$ is well-known (see Theorem~\ref{th: Nyberg's p-ary distribution}), while the situation for $m>1$ has up until now not been determined.

For bent functions from $\F_{p}^n$ to $\F_{p}^m$ one can use the results from the previous section to derive very strong conditions on the preimage distributions if certain spectral conditions are satisfied. Note that this again covers the important cases of Boolean vectorial bent functions, $p$-ary bent functions with regular component functions.

\begin{theorem} \label{thm:low_m}
	Let $F \colon \F_{p}^n \rightarrow \F_{p}^m$ be a bent function such that $W_F(b,0)=\epsilon p^{n/2}\zeta_p^{r_b}$ with $\epsilon \in \{1, i\}$ for all $b \in \F_{p^m}^*$. Then, the preimage set sizes are $X_i=p^{n-m}+p^{n/2-m}(pT_i-1)$ for all $i \in \{1,\dots,p^m\}$ where the $T_i$ are integers satisfying the two equations
	\begin{align}
			\sum_{i=1}^{p^m} T_i^2&=p^{2m-2}\label{eq:5_reg} \\
			\sum_{i=1}^{p^m} T_i&=p^{m-1}. \label{eq:6_reg}    
	\end{align}
\end{theorem}
\begin{proof}
	Assume Equations~\eqref{eq:3} and \eqref{eq:4} hold for $H_1,\dots,H_{p^m}$. By Theorem~\ref{thm:regular_constraints}, we have $H_i=p^{n/2}-p^{n/2-m}(pk_i+1)=p^{n/2-m}(p^m-pk_i-1)$ for $0 \leq k_i \leq \frac{p^m-1}{p-1}$ and thus $p^{n/2-m}|H_i$. Write $H_i'=\frac{H_i}{p^{n/2-m}}$. Plugging this into Equations~\eqref{eq:3} and \eqref{eq:4} yields 
		\begin{align*}
			\sum_{i=1}^{p^m} (H'_i)^2&=p^{2m}-p^m \\
			\sum_{i=1}^{p^m} H'_i&=0.   
	\end{align*}
	Observe that $H_i'\equiv -1 \pmod p$  and set $H_i'=pT_i-1$. Plugging this into the equations above yields the desired equations on the $T_i$. Retracing the substitutions yields $X_i=p^{n-m}+H_i=p^{n-m}+p^{n/2-m}H_i'=p^{n-m}+p^{n/2-m}(pT_i-1)$. 
\end{proof}
The extremal distribution $(+)$ solves Equations~\eqref{eq:5_reg} and~\eqref{eq:6_reg} with
$T_1=p^{m-1}$ and $T_i=0$ for all $i>0$ (recall that the $(-)$ case cannot occur here if $p$ is odd by Theorem~\ref{thm:walsh}). For $p=2$, the solution $T_1=-2^{m-1}+1$, $T_i=1$ for $i>0$ yields the $(-)$ case.

\begin{theorem} \label{thm:low_m_weakly}
	Let $F \colon \F_{p}^n \rightarrow \F_{p}^m$ be a bent function such that $W_F(b,0)=\epsilon p^{n/2}\zeta_p^{r_b}$ with $\epsilon \in \{-1, -i\}$ for all $b \in \F_{p^m}^*$. Then, the preimage set sizes are $X_i=p^{n-m}+p^{n/2-m}(1-pT_i)$ for all $i \in \{1,\dots,p^m\}$ where the $T_i$ are integers satisfying the two equations
	\begin{align}
			\sum_{i=1}^{p^m} T_i^2&=p^{2m-2}\label{eq:5} \\
			\sum_{i=1}^{p^m} T_i&=p^{m-1}. \label{eq:6}    
	\end{align}
\end{theorem}
\begin{proof}
	Assume Equations~\eqref{eq:3} and \eqref{eq:4} hold for $H_1,\dots,H_{p^m}$. By Theorem~\ref{thm:regular_constraints} we have $H_i=p^{n/2-m}(pk_i+1)-p^{n/2}=p^{n/2-m}(pk_i-p^m+1)$ for $0 \leq k_i \leq \frac{p^m-1}{p-1}$ and thus $p^{n/2-m}|H_i$. Write $H_i'=\frac{H_i}{p^{n/2-m}}$. Plugging this into Equations~\eqref{eq:3} and \eqref{eq:4} yields 
		\begin{align*}
			\sum_{i=1}^{p^m} (H'_i)^2&=p^{2m}-p^m \\
			\sum_{i=1}^{p^m} H'_i&=0.   
	\end{align*}
	Observe that $H_i'\equiv 1 \pmod p$  and set $H_i'=-pT_i+1$. Plugging this into the equations above yields the desired equations on the $T_i$. Retracing the substitutions yields $X_i=p^{n-m}+H_i=p^{n-m}+p^{n/2-m}H_i'=p^{n-m}+p^{n/2-m}(-pT_i+1)$. 
\end{proof}

Here, the extremal distribution $(-)$ solves Equations~\eqref{eq:5} and~\eqref{eq:6} with
$T_1=p^{m-1}$ and $T_i=0$ for all $i>0$ (here the $(+)$ case cannot occur if $p$ is odd).  For $p=2$, the solution $T_1=-2^{m-1}+1$, $T_i=1$ for $i>0$ yields the $(+)$ case.

For $p=2$ and low values of $m$ (and arbitrary $n$) we can use these results to determine all possible preimage distributions of vectorial Boolean bent functions. Since the conditions on the Walsh transform hold trivially in this case, we can use both Theorem~\ref{thm:low_m} and Theorem~\ref{thm:low_m_weakly} and get the same results. Note also that the symmetry (with respect to the signs) from Equations~\eqref{eq:3} and \eqref{eq:4} is still visible in the binary case: Indeed, if $\{T_1,\dots,T_{2^m}\}$ is a valid solution then so is $\{-T_1+1,-T_2+1,\dots,-T_{2^m}+1\}$. Note that all $T_i$ have the same parity by Theorem~\ref{thm:boolean_constraints}. By the symmetry above, we can concentrate on even $T_i$ since the solutions with odd $T_i$ are covered by the symmetry above. 

While this direct connection does not exist anymore in the $p$-ary case, the solutions here still come in pairs since a solution of the Equations~\eqref{eq:5_reg} and \eqref{eq:6_reg} yield different distributions for the two Theorems~\ref{thm:low_m} and Theorem~\ref{thm:low_m_weakly} (while the distributions in the Boolean case overlap). 

\begin{theorem}\label{thm:m2}
    Let $F \colon \F_2^n \rightarrow \F_2^2$ be a bent function. Then there are only two possible preimage distributions which are exactly the two extremal distributions $(+)$ and $(-)$. Both distributions occur for any even $n\geq 4$.
\end{theorem}
\begin{proof}
    From Theorem~\ref{thm:low_m} we get the equations in the $T_i$: 
    		\begin{equation*}
			\sum_{i=1}^{4} T_i^2=4 ,\;\; \sum_{i=1}^{4} T_i=2.   
	\end{equation*}
		It is easy to see that the only possible integer solutions are (up to permutation of the $T_i$): $T_1=2$, $T_2=T_3=T_4= 0$ and $T_1=-1$, $T_2=T_3=T_4=1$. These distributions belong to the two extremal distributions $(+)$ and $(-)$. Conversely, Theorem~\ref{th:existence} shows that vectorial bent functions $F\colon\F_2^n\to\F_2^2$ of types $(+)$ and $(-)$ exist for all $n \geq 4$.
\end{proof}
This allows us to deduce a simple corollary.
\begin{corollary}
	Let $F \colon \F_2^n \rightarrow \F_2^m$ with $m\in \{1,2\}$ be a plateaued function. Then $F$ is bent if and only if $F$ is of type $(-)$ or $(+)$.
\end{corollary}
\begin{proof}
	Follows from Theorem~\ref{thm:walsh} together with Remark~\ref{rem:boolean} and Theorem~\ref{thm:m2}, respectively.
\end{proof}

\begin{theorem}\label{thm:m3}
    	Let $F \colon \F_{2}^n \rightarrow \F_{2}^3$ be a bent function. Then there are only four possible preimage distributions which are the distributions $(+)$ and $(-)$, and the distributions with preimage set sizes $X_i=2^{n-3}+2^{n/2-3}(2T_i-1)$ where 
    	\begin{align*}
    	     &T_1=-2, T_2=T_3=T_4=2, T_5=\dots=T_8=0 \text{, or}\\
    	      &T_1=3, T_2=T_3=T_4=-1, T_5=\dots=T_8=1.
    	\end{align*}
\end{theorem}
\begin{proof}
    From Theorem~\ref{thm:low_m} we get the equations in the $T_i$: 
    		\begin{equation*}
			\sum_{i=1}^{8} T_i^2=16 ,\;\; \sum_{i=1}^{8} T_i=4.   
	\end{equation*}
	From Theorem~\ref{thm:boolean_constraints} and the discussion above we can concentrate on even $T_i$, getting the odd solutions via symmetry. This effectively makes the set of equations even easier. 
		With little effort, one gets the solutions. Up to a permutation, we only get the two solutions belonging to $(+)$ and $(-)$: $T_1=4$, $T_2=T_3=\dots=T_8= 0$ and $T_1=-3$, $T_2=T_3=\dots=T_8=1$, as well as the even solution $T_1=-2$, $T_2=T_3=T_4=2$, $T_5=\dots=T_8=0$ and its ``symmetric'' solution $T_1=3$, $T_2=T_3=T_4=-1$, $T_5=\dots=T_8=1$. 
\end{proof}

\begin{remark}
    We have checked with a computer program that by adding linear functions to representatives of the equivalence classes of vectorial bent functions $F\colon\F_2^6\to\F_2^3$ from~\cite{polujan2020design}, it is possible to obtain all four distributions in Theorem~\ref{thm:m3}. We conjecture that all four preimage distributions from Theorem~\ref{thm:m3} occur for all vectorial bent functions $F\colon\F_2^n\to\F_2^3$ with even $n\geq 6$. Of course, it was only left to prove the existence of the two non-extremal distributions since the two extremal distributions are covered by Theorem~\ref{th:existence}.
\end{remark}

    In the following statement, we also analyse value distributions of vectorial Boolean bent functions $F \colon \F_{2^n} \rightarrow \F_{2^4}$.
\begin{theorem}\label{thm:m4}
    	Let $F \colon \F_{2^n} \rightarrow \F_{2^4}$ be a bent function. Then there are only 14 possible preimage distributions which are the distributions with preimage set sizes $X_i=2^{n-4}+2^{n/2-4}(2T_i-1)$ where the $T_i$ are one of the following:
    \begin{enumerate}
        \item $T_1=-6, T_2=\dots=T_{9}=0,T_{10}=\dots=T_{16}=2,$
        \item $T_1=-4, T_2=T_{3}=-2,T_{4}=\dots=T_{9}=0,T_{10}=\dots=T_{15}=2, T_{16}=4$,
        \item $T_1=-4, T_2=\dots=T_{13}=0,T_{14}=\dots=T_{16}=4,$
        \item $T_1=\dots=T_{6}=-2,T_{7}=\dots=T_{16}=2,$
        \item $T_1=\dots=T_{4}=-2,T_{5}=\dots=T_{10}=0,T_{11}=\dots=T_{14}=2, T_{15}=T_{16}=4,$
        \item $T_1=\dots=T_{3}=-2,T_{4}=\dots=T_{11}=0,T_{12}=\dots=T_{15}=2, T_{16}=6$,
         \item $T_1=\dots=T_{15}=0,T_{16}=8,$
        \item The symmetric solutions $$\{-T_1+1,-T_2+1,\dots,-T_{16}+1\}$$ for all 7 solutions above.
    \end{enumerate}
\end{theorem}
\begin{proof}
    We again solve the equations on the $T_i$, focusing only on even $T_i$. The solutions for odd $T_i$ follow again using the symmetry $\{T_1,\dots,T_{16}\} \mapsto \{-T_1+1,-T_2+1,\dots,-T_{16}+1\}$. Setting $T_i=2T_i'$, we get
        \begin{equation*}
			\sum_{i=1}^{16} T_i'^2=16 ,\;\; \sum_{i=1}^{16} T_i'=4.   
	\end{equation*}
	 {
		Clearly, $|T_i'| \leq 4$ for any $i$. Define $c_j$ for $-4 \leq j \leq 4$ such that $c_j$ is the number of times that $T_i=j$ in the previous set of equations. We can then rephrase these two equations as three easier equations, namely
		\[\sum_{j=-4}^{4} c_j=16 ,\;\; \sum_{j=-4}^{4} j^2c_j=16 ,\;\;\sum_{j=-4}^{4} jc_j=4. \]
		Here, the first equation corresponds to the fact that there are $1 \leq i \leq 16$ in the original set of equations and the second and third equation capture the previous equations. The solutions of these three equations can be found quite easily by computer (or with a bit of effort by hand), keeping in mind that the $c_j$ are non-negative integers. Retranslating these in terms of $T_i$, we get in total 14 solutions. 7 of them correspond to the ones given in the statement of the theorem, as well as the following 7 additional ones:
	}
	

 \begin{enumerate}[i)]
     \item $T_1=T_2=-4, T_3=\dots=T_{8}=0,T_{9}=\dots=T_{16}=2,$\label{case: i}
     \item  $T_1=-4, T_2=\dots=T_{4}=-2,T_5=\dots=T_7=0,T_{8}=\dots=T_{16}=2,$\label{case: ii}
     \item  $T_1=-4, T_2=-2,T_3=\dots=T_{11}=0,T_{12}=\dots=T_{14}=2,T_{15}=T_{16}=4$,\label{case: iii}
     \item  $T_1=-4, T_2=\dots=T_{12}=0,T_{13}=\dots=T_{15}=2,T_{16}=6$,\label{case: iv}
     \item  $T_1=\dots=T_{5}=-2,T_{6}=\dots=T_{8}=0,T_9=\dots=T_{15}=2,T_{16}=4$,\label{case: v}
     \item  $T_1=\dots=T_{3}=-2,T_{4}=\dots=T_{12}=0,T_{13}=2,T_{14}=\dots=T_{16}=4$,\label{case: vi}
     \item  $T_1=T_{2}=-2,T_{3}=\dots=T_{13}=0,T_{14}=2,T_{15}=4,T_{16}=6$.\label{case: vii}
 \end{enumerate}
Now, we show that all 7 distributions from the list cannot occur for bent functions.
 
By Theorem~\ref{thm:regular_constraints}, we have $k_a=|\{b \in (\F_{2^4})^* \colon W_F(b,0)=-2^{n/2}\cdot (-1)^{\Tr(ab)}\}|=8-T_i$ for some $a \in \F_2^4$ and some $1 \leq i \leq 16$. We can pick without loss of generality (by shifting the function by a constant) the $i$ that corresponds to $k_0$. Set $K=\{b \in \F_{2^4}^* \colon W_F(b,0)=-2^{n/2}\}$. 

Let us first assume that $k_0=12$, corresponding to a value of $T_i=-4$, so $W_F(b,0)=-2^{n/2}$ for $12$ choices of $b$ and $W_F(b,0)=2^{n/2}$ for 3 choices of $b$. Then $k_a \in \{4,6,8,10\}$ for $a \neq 0$, depending on $\Tr(ab)$ for $b \in K$. For instance,  $k_a=4$ iff all $b$ with $\Tr(ab)=1$ are contained in $K$ and $k_a=6$ iff precisely 7 $b$ with $\Tr(ab)=1$ are contained in $K$. In particular, it is impossible that $k_a=12$ for $a \neq 0$. This means that there can only be at most one $i$ such that $T_i=-4$, excluding the case~\ref{case: i}.

We can write $\{x,y,z\}=\F_{2^4}^*\setminus K$ and let $H$ be a hyperplane of $\F_{2^4}$ containing $x,y,z$. Then $|H \cap K|=4$ and (denoting $\overline{H}=\F_{2^4} \setminus H$) clearly $|\overline{H} \cap K|=8$. By the considerations above this means that there exists an $a$ such that $k_a=4$, corresponding to a $T_j=4$. We conclude that if one $T_i=-4$ there also has to exist a $j$ with $T_j=4$. This excludes the cases~\ref{case: ii} and~\ref{case: iv}.

Assume now that we have another value of $j$ such that $T_j=4$ (still $T_i=-4$), corresponding to a $k_a=4$. Then (as outlined above) there is a hyperplane $H$ with $|H \cap K |=4$ and $H$ must contain $x,y,z$. If $x,y,z$ are linearly independent, this $H= \langle x,y,z \rangle$ is uniquely determined. If $x+y=z$ then there are precisely $3$ choices for $H$. We conclude that if $T_i=-4$ there are either one or three $j$ such that $T_j=4$. This excludes the case~\ref{case: iii}. 

Let us now deal with the last three cases~\ref{case: v}, \ref{case: vi} and~\ref{case: vii}. All have in common that there is no $i$ with $T_i=-4$ but there is an $i$ with $T_i=4$. Let us thus assume that $k_0=4$, which corresponds to a $T_i=4$. We can thus set $K=\{x,y,z,w\}$. We have $k_a=12$ (corresponding to $T_j=-4$) if and only if $\Tr(ax)=\Tr(ay)=\Tr(az)=\Tr(aw)=0$ for some $a \neq 0$, i.e., $K$ is contained in a hyperplane. Since $T_j=-4$ does not occur, this means $K$ is not contained in a hyperplane which means that $x,y,z,w$ are linearly independent. Similarly, we have $k_a=4$ (corresponding to a second $T_j=4$) for $a \neq 0$ if and only if  $\Tr(ax)=\Tr(ay)=\Tr(az)=\Tr(aw)=1$ for some $a$, i.e., $K$ is contained in an affine hyperplane. This occurs if and only if $x+y,x+z,x+w$ are contained in the hyperplane $H_a=\{b \colon \Tr(ab)=0\}$ and $x \notin H_a$. But $x+y,x+z,x+w$ are linearly independent, so $H_a= \langle x+y,x+z,x+w \rangle$ is uniquely determined. This means that there is at most one second $j$ (next to $i$) such that $T_j=4$, excluding case~\ref{case: vi}. If there is no $j\neq i$ such that $T_j=4$ then (by the argument above) $K$ is not contained in an affine hyperplane, in other words, $|H \cap K| \geq 1$ for all hyperplanes $H$. But set again $H=\langle x+y,x+z,x+w \rangle$ and since $x,y,z,w$ are linearly independent, we have $H \cap K = \emptyset$, yielding a contradiction. We conclude that there must be exactly one second $j\neq i$ such that $T_j=4$ . This excludes the last remaining two cases~\ref{case: v} and~\ref{case: vii} .

It is easy to observe that the symmetric cases $ \{-T_1+1,-T_2+1,\dots,-T_{16}+1\}$ to the 7 cases we just excluded also cannot occur. Indeed, we can repeat the arguments above, just replacing
\begin{equation*}
    \begin{split}
        k_a=&|\{b \in \F_{2^4}^* \colon W_F(b,0)=-2^{n/2}\cdot (-1)^{\Tr(ab)}\}| \mbox{ with} \\
        k'_a=&|\{b \in \F_{2^4}^* \colon W_F(b,0)=2^{n/2}\cdot (-1)^{\Tr(ab)}\}|,
    \end{split}
\end{equation*}
i.e., a change of a sign.
\end{proof}

We have checked by computer that all the 14 cases in Theorem~\ref{thm:m4} do in fact occur for vectorial Boolean bent functions in $8$ variables. 

\begin{proposition}\label{prop: (8,4)-bent}
    The preimage distributions of vectorial bent functions $F\colon\F_2^8\to\F_2^4$ are precisely the 14 distributions given in Theorem~\ref{thm:m4}.
\end{proposition}
\begin{verification}
Consider the following vectorial bent function $F\colon\F_2^8\to\F_2^4$ (this is the first function in the list~\cite{PolPot_Dataset} obtained in~\cite{PolujanBFA22}), which is given by its algebraic normal form as follows:
\begin{equation}\label{eq: (8,4)-bent with all distros}
    F(x_1,\ldots,x_8)=\scalebox{0.9}{$\begin{pmatrix}x_1 x_5 + x_2 x_6 + x_3 x_7 + x_4 x_8\\ 
 x_1 x_3 + x_1 x_4 + x_3 x_4 + x_2 x_5 + x_4 x_5 + x_3 x_6 + x_4 x_6 + x_1 x_7 + 
  x_3 x_7 + x_4 x_7 + x_2 x_8\\ 
 x_1 x_5 + x_3 x_5 + x_4 x_5 + x_2 x_6 + x_3 x_6 + x_2 x_7 + x_1 x_8 + x_2 x_8\\ 
 x_1 x_3 + x_1 x_4 + x_3 x_5 + x_2 x_7 + x_5 x_7 + x_1 x_8 + x_6 x_8 \end{pmatrix}$}.
\end{equation}
With a help of a computer program, it is possible to check that by adding random linear functions to the bent function $F$ defined in~\eqref{eq: (8,4)-bent with all distros}, one soon gets all possible distributions given in the statement of Theorem~\ref{thm:m4}.
\end{verification}
\begin{remark}
    In view of Proposition~\ref{prop: (8,4)-bent} and Theorem~\ref{thm:m4}, we conjecture that all 14 preimage distributions from Theorem~\ref{thm:m4} occur for all vectorial bent functions  $F\colon\F_2^n\to\F_2^4$  with even $n\geq 8$. 
\end{remark}

Finally, we demonstrate that in some cases it is possible to determine possible value distributions of bent functions without the exact knowledge of spectral properties of the considered functions. This allows us to extend the proof of Theorem~\ref{thm:m2} to arbitrary groups of sizes $|G|=2^n$ and $|H|=4$.

\begin{theorem}\label{thm:jc}
	Let $G$ and $H$ be two finite groups with $|G|=2^n$ and $|H|=4$. Let $F \colon G \rightarrow H$ be a perfect nonlinear function. Then there are only two possible preimage distributions which are exactly the two extremal distributions from Theorem~\ref{thm:CS general}.
\end{theorem}
\begin{proof}
	Assume Equations~\eqref{eq:3} and \eqref{eq:4} hold for $H_1,\dots,H_{4}$ and $n>4$. Then
	\[\sum_{i=1}^{4} H_i^2\equiv 0 \pmod 8.\]
	Observe that $H_i^2$ is $0,1,4 \pmod 8$ and $H_i^2\equiv 1 \pmod 8$ if and only if $H_i$ is odd. This implies that the $H_i$ are all even, say $H_i=2H_i'$. Then Equations~\eqref{eq:3} and \eqref{eq:4} become
		\begin{align*}
		\sum_{i=1}^{4} (H_i')^2&=2^{n-2}-2^{(n-2)-2} \\
			\sum_{i=1}^{4} H_i'&=0. 
	\end{align*}
	$H_1,\dots,H_4$ is thus a solution of Equations~\eqref{eq:3} and  \eqref{eq:4} for $n$ if and only if $H_1',\dots,H_4'$ is a solution of  Equations~\eqref{eq:3}, \eqref{eq:4} for $n-2$. We can continue this procedure until $n=4$, since in this case $2^{n-2}$ is no longer divisible by 8. We arrive at 
			\begin{align*}
		\sum_{i=1}^{4} (H_i')^2&=2^{4}-2^{4-2}=12 \\
			\sum_{i=1}^{4} H_i'&=0. 
	\end{align*}
	One can see that the only possible solutions  (up to permutation of the $H_i$) are $H_1=H_2=H_3= \pm 1$, $H_4=\mp 3$. This means that the only possible solutions for the general case are $H_1=H_2=H_3= \pm 2^{n/2-2}$, $H_4=\mp 3 \cdot 2^{n/2-2}$. These correspond to preimage set sizes $X_1=X_2=X_3=2^{n-2} \pm 2^{n/2-2}$, $X_4=2^{n-2} \mp 3\cdot 2^{n/2-2}$.
\end{proof}

\section{Value distributions of planar functions}\label{section: 6 Preimage sets of planar functions}

In this section, we discuss the particularly interesting case of planar functions, i.e., vectorial bent functions with $p$ odd and $n=m$. Planar functions have important applications both for difference sets (as they give rise to examples of skew Hadamard difference sets that are inequivalent to Paley difference sets~\cite{ding2006family}) and commutative semifields (see, e.g.,~\cite{golouglu2022exponential}) which play an important role in finite geometry. 

The image sets of planar functions were considered in~\cite{KyureghyanP2008} and~\cite{Coulter2014} where lower and upper bounds (respectively) of the image set sizes of planar functions were derived. Using the tools we developed in the previous sections we are able to unify these results and give an alternative proof of one of the main results in~\cite[Theorem 2]{KyureghyanP2008} as well as~\cite{Coulter2014}, while giving (for the upper bound) more precise information on the preimage set distribution that occurs in the extremal cases. Recall that we call a function $F \colon \F_{p^n} \rightarrow \F_{p^n}$ with $p$ odd $2$-to-$1$ on $\F_{p^n}$ if one unique element has one preimage and $(p^n-1)/2$ elements have $2$ preimages.

\begin{proposition} \label{prop:planar_imageset}
	Let $F \colon \F_{p^n} \rightarrow \F_{p^n}$ be a planar function. Then
	\[\frac{p^n+1}{2} \leq |\image(F)| \leq p^n-\frac{1}{2}(\sqrt{4p^n-3}-1).\]
	The lower bound is satisfied with equality if and only if $F$ is $2$-to-$1$ on $\F_{p^n}$ and the upper bound is satisfied with equality if and only if all but one element in the image set have a unique preimage.
\end{proposition}
\begin{proof}
	We consider Equations~\eqref{eq:3} and~\eqref{eq:4}. Since $H_i=1+X_i$ we have $H_i \geq -1$. Set $k=|\{i \colon H_i=-1\}|$ and let $H_1=\dots=H_k=-1$, $H_{k+1}=\dots=H_{k+s}=0$ and $H_i>0$ for $i>k+s$. Clearly, $|\image(F)|=p^n-k$. We get from Equations~\eqref{eq:3} and~\eqref{eq:4}
		\begin{align}
		\sum_{i=k+s+1}^{p^n} H_i^2&=p^n-1-k \label{eq:eq1withk} \\
			\sum_{i=k+s+1}^{p^n} H_i&=k. \label{eq:eq2withk}
	\end{align}
	By the Cauchy-Schwarz inequality, we get $k^2\leq (p^n-1-k)(p^n-k-s)$
	with equality if and only if all $H_i$ with $i>k+s$ are equal. From Proposition~\ref{prop:imageset} we know the possible minimum value is $k=\frac{p^n-1}{2}$. Plugging this in yields $\frac{p^n-1}{2} \leq \frac{p^n+1}{2}-s$, only leaving $s \in \{0,1\}$ as possibilities. If $s=0$, then Equation~\eqref{eq:eq2withk} cannot be satisfied as all $H_i \geq 1$ and the sum contains $\frac{p^m+1}{2}=k+1$ terms. Thus $s=1$ and it is easy to see that $H_i=1$ for all $i>k+s$, meaning that $X_i=2$ for these $i$ and $F$ is $2$-to-$1$.
	
	Let us consider the upper bound now. Define $M_i$ as the number of elements $y \in \image(F)$ with precisely $i$ preimages. Then (by Proposition~\ref{prop:CS_1}) $p^n+p^{n-m}(p^n-1)=\sum_{i=1}^k i^2M_i$ where $r$ is the maximum preimage set size of $F$. Note that $\sum_{i=1}^r iM_i=p^n$, so 
\[p^{n-m}(p^n-1)=\sum_{i=1}^r i(i-1)M_i \leq r \sum_{i=1}^r (i-1)M_i,\]
with equality if and only if $M_i=0$ for all $2 \leq i <r$. Then
\begin{align*}
	|\image(F)| &= \sum_{i=1}^r M_i = \sum_{i=1}^r iM_i - \sum_{i=1}^r (i-1)M_i \\
	&=p^n-\sum_{i=1}^r (i-1)M_i \leq p^n-\frac{p^{n-m}(p^n-1)}{r},
\end{align*}
still with equality if and only if $M_i=0$ for all $2 \leq i <r$. Clearly, the bound is best if $r$ is maximal, i.e., if the maximum preimage set size is as high as possible, which means maximizing one $H_i$ in Equation~\eqref{eq:eq1withk}. This occurs if Equations~\eqref{eq:eq1withk} and~\eqref{eq:eq2withk}  only have one term on the left-hand side, which yields $H_i=k$,$H_i^2=p^n-1-k$, i.e., $k^2+k=p^n-1$ which has the positive solution $k=\frac{1}{2}(\sqrt{4p^n-3}-1)$, leading to $|\image(F)|=p^n-\frac{1}{2}(\sqrt{4p^n-3}-1)$. Equality is achieved if and only if there is one element with $\frac{1}{2}(\sqrt{4p^n-3}-1)+1$ preimages (since $X_i=1+H_i=1+k$) and all other elements in the image set have a unique preimage.
\end{proof}
\begin{remark}
    We are not aware of any planar functions attaining the upper bound. A necessary condition is that $4p^n-3=4(p^n-1)+1=8(\frac{p^n-1}{2})+1$ is a square. As already observed by Coulter and Senger in a slightly different context~\cite{Coulter2014}, this is the case if and only if $\frac{p^n-1}{2}$ is a triangular number, i.e., a number of the form $u(u-1)/2$. This would mean $p^n-1=u(u-1)$, i.e., this occurs if and only if $p^n-1$ is the product of two consecutive numbers. This can clearly never occur if $n$ is even since then $p^n-1=(p^{n/2}-1)(p^{n/2}+1)$ and it is clear that $p^n-1$ is not the product of two consecutive numbers. For $n$ odd, this can however happen, simple examples include $7-1=2 \cdot 3$ and $7^3-1=18\cdot 19$. 
\end{remark}

Note that many examples of planar functions satisfying the lower bound are known, in fact, all planar Dembowski-Ostrom polynomials (i.e., polynomials of the form $\displaystyle F(x)=\sum_{i,j=0}^{n-1}a_{i,j}x^{p^i+p^j}$, where $a_{i,j}\in\F_{p^n}$ and $x\in\F_{p^n}$) are necessarily $2$-to-$1$ which is well known~\cite[Corollary 1]{KyureghyanP2008} {, as well as the (non-Dembowski-Ostrom) Coulter-Matthews planar function}. We add a short proof of the Dembowski-Ostrom case using our techniques here as well. 

\begin{proposition}\label{prop:planar_even}
	Let $F \colon \F_{p^n} \rightarrow \F_{p^n}$ be a planar function such that $F(x)=F(-x)$ for all $x \in \F_{p^n}$. Then $F$ is $2$-to-$1$. 
\end{proposition}
\begin{proof}
	We again consider Equations~\eqref{eq:3} and~\eqref{eq:4}. We have $F(x)=F(-x)$ for all $x \in \F_{p^n}$, so the number of preimages is odd only precisely once. Then the $H_i=1+X_i$ are all odd with exactly one exception. By  Equation~\eqref{eq:3}, one $H_i$ has to be 0 and all other satisfy $H_i^2=1$ and then by Equation~\eqref{eq:4}, necessarily $\frac{p^n-1}{2}$ choices of $H_i$ are $1$ and the same holds for $-1$. Keeping in mind that $X_i=1+H_i$, we conclude that $F$ is $2$-to-$1$.
\end{proof}

Again, the $2$-to-$1$ property (i.e., information on the preimage size distribution) is enough to force planarity as long as we only consider plateaued functions. This is in many ways surprising since both the $2$-to-$1$ functions as well as plateaued functions seem to be much more prevalent than planar functions. The result and proof idea is an analogue of a similar result for $3$-to-$1$ almost perfect nonlinear functions achieved in~\cite{Koelsch2023}.

\begin{theorem} \label{thm:planarplateaued}
	Let $F \colon \F_{p^n} \rightarrow \F_{p^n}$ be a plateaued, $2$-to-$1$ function. Then $F$ is planar.
\end{theorem}
\begin{proof}
    Denote for simplicity by $\chi(x)=\zeta_p^{\Tr(x)}$ for $x\in\F_p^n$ the canonical additive character. Let us first show that $F$ does not have any balanced component functions, i.e., $W_F(b,0)\neq 0$ for $b \in \F_{p^n}^*$. Since $F$ is $2$-to-$1$ we have
	\[W_F(b,0) = \sum_{x \in \F_{p^n}} \chi(bF(x)) = \chi(bx_0) + 2\sum_{x \in M} \chi(bx),\]
	where $x_0$ is the element that has $1$ preimage and $M$ is the set of elements with $2$ preimages. We have thus clearly $W_F(b,0) \equiv \chi(bx_0) \not\equiv 0 \pmod 2$, in particular $W_F(b,0) \neq 0$. 
	
	Now set $N_k=|\{b \in \F_{p^n}^* \colon |W_F(b,0)|=p^{(n+k)/2}\}|$. Since $W_F(b,0)\neq 0$, we infer that $N_k$ is the number of plateaued component functions with amplitude $k$, so
	\begin{equation}
		\sum_{k \geq 0 } N_k = p^n-1.
	\label{eq:sumcomponents}
	\end{equation}
	Since $F$ is $2$-to-$1$, we have
		\[\frac{1}{p^n}\sum_{b \in \F_{p^n}}\sum_{x_1,x_2 \in \F_{p^n}} \chi(b(F(x_1)-F(x_2))=1+2(p^n-1)=2p^n-1.\]
	On the other hand, 
		\begin{align*}
		\frac{1}{p^n}\sum_{b \in \F_{p^n}}\sum_{x_1,x_2 \in \F_{p^n}} \chi(b(F(x_1)-F(x_2)) &= p^n+\frac{1}{p^n}\sum_{b \in \F_{p^n}^*}\sum_{x_1,x_2 \in \F_{p^n}} \chi(b(F(x_1)-F(x_2)) \\
		&=p^n+\frac{1}{p^n}\sum_{b \in \F_{p^n}^*}\sum_{x_1,x_2 \in \F_{p^n}} \chi(bF(x_1))\overline{\chi(bF(x_2))}\\
		&=p^n+\frac{1}{p^n}\sum_{b \in \F_{p^n}^*}|W_F(b,0)|^2 \\
		&= p^n + N_0+pN_1+p^2N_2+ \dots
	\end{align*}
	We thus infer
	\[p^n-1 = N_0+pN_1+p^2N_2+ \dots \]
	and with Equation~\eqref{eq:sumcomponents} $N_0=p^n-1$ and $N_k=0$ for all $k>0$, so all component functions of $F$ are bent and $F$ is planar.
\end{proof}

This allows us to state the following corollary. Note that this is a strict generalization of one of the main results in~\cite[Theorem 2.3.]{weng2012further} which showed the statement only for DO polynomials, which are a subclass of plateaued functions satisfying $F(x)=F(-x)$. 
\begin{corollary} \label{cor:2to1}
	Let $F \colon \F_{p^n} \rightarrow \F_{p^n}$ be a plateaued function such that $F(x)=F(-x)$ for all $x \in \F_{p^n}$. Then $F$ is planar if and only if $F$ is $2$-to-$1$.
\end{corollary}
\begin{proof}
	Follows from Theorem~\ref{thm:planarplateaued} and Proposition~\ref{prop:planar_even}.
\end{proof}
Corollary~\ref{cor:2to1} is indeed a generalization from the DO case since plateaued planar functions that are not DO polynomials do in fact exist, an example is the Coulter-Matthews planar monomial~\cite{coulter1997planar}.

In particular, for monomials, proving planarity can then essentially be reduced to proving the plateaued condition (recall that planar functions cannot be bijective by Proposition~\ref{prop:planar_imageset}). 
\begin{corollary} \label{cor:planar_monom}
    Let $F \colon \F_{p^n} \rightarrow \F_{p^n}$ be a plateaued monomial $F=x^d$. Then $F$ is planar if and only if $\gcd(d,p^n-1)=2$. 
\end{corollary}
\begin{proof}
	Follows from Theorem~\ref{thm:planarplateaued} and Proposition~\ref{prop:planar_imageset}.
\end{proof}

\section{Conclusion and open problems}\label{section: 7 Conclusion and open problems}

In this paper, we systematically developed the theory of value distributions for perfect nonlinear functions. Particularly,  we provided a purely combinatorial framework for checking the equivalence of perfect nonlinear functions $F\colon\F_p^n\to\F_p^m$ in terms of the value distributions. Moreover, we were able to describe all possible value distributions for several large classes of perfect nonlinear functions. In general, however, it seems to be a very difficult problem to determine all possible and impossible value distributions completely  since the techniques we used rely on precise spectral information and solving systems of quadratic Diophantine equations, whose possible number of solutions grows with the increasing output group's order. To conclude, we believe that answering the following questions (in addition to already mentioned open problems and conjectures in the previous sections) will help to provide a better understanding of perfect nonlinear functions, and, more generally, cryptographically significant classes of functions.
\begin{enumerate}
    \item The theory of value distributions of perfect nonlinear functions $F\colon\F_p^n\to\F_p^m$ (and in general, on arbitrary groups) developed in this article actually only hinges on the fact that $F(x)=F(y)$ has precisely $p^n+p^{n-m}(p^n-1)$ solutions. This is, however, not unique for bent functions. Are there other functions of interest with this property? It would be also interesting to investigate in a similar manner other classes of cryptographically significant functions $F\colon\F_p^n\to\F_p^m$, for instance,  plateaued and differentially uniform functions.
    \item So far, the known constructions of almost balanced perfect nonlinear functions $F\colon\F_p^n\to\F_p^m$ are mostly dominated by the ($+$) type constructions. It would be interesting to provide more primary constructions of almost balanced bent functions of the type ($-$), especially in the $p$ odd case. 
    \item Besides the ``direct sum'' construction, there exist many other secondary constructions of bent functions, which can be described without loss of generality in the following form $F(F_1(x_1),\ldots,F_k(x_k))$ where $F_i\colon\F_p^{n_i}\to\F_p^m$ are perfect nonlinear functions, and $F$ is some function. It would be interesting to provide initial conditions on bent functions $F_1,\ldots,F_k$, which guarantee that the obtained perfect nonlinear functions of the form $(x_1,\ldots,x_k)\in\F_p^{n_1}\times\cdots\times \F_p^{n_k}\mapsto F(F_1(x_1),\ldots,F_k(x_k))$ are almost balanced of $(+)$ and $(-)$ types, and hence are inequivalent  (in the $p$ odd case).
    \item Many of the known constructions of vectorial bent functions, whose preimage sets can be used to construct partial difference sets, are, in fact, almost balanced, see~\cite{CMP2021,Wang2023}. In this regard, it is natural to find other constructions of almost balanced bent functions, which give rise to partial difference sets.
    \item Our computer experiments show that in the Boolean case, it is possible to obtain both extremal value distributions by adding suitable linear functions from a single given vectorial Boolean bent function, opposite to the $p$-ary case. It would be interesting to investigate whether for every vectorial Boolean bent function one can add linear functions to obtain both $(+)$ and $(-)$ extremal value distributions, or whether there exist vectorial Boolean functions which cannot reach both extremal value distributions with a help of the addition of suitable linear functions.
    \item We showed that many primary and secondary constructions of perfect nonlinear functions appear to have extremal value distributions, which implies that the ``almost balanced'' property unifies many algebraically different constructions. With this observation, it is essential to construct more perfect nonlinear functions in a purely combinatorial manner, using the ``almost balanced'' property.
\end{enumerate}

\section*{Acknowledgements}
The ideas in this article were partially developed while both authors visited Gohar Kyureghyan at the University of Rostock in late September 2022. We are grateful to her for the invitation, fruitful discussions and excellent working conditions.

We would also like to thank Jan-Christoph Schlage-Puchta, who kindly suggested the idea of considering  ``deviations from the mean'' as well as the proof of Theorem~\ref{thm:jc}, which were then later developed to the results in Section~\ref{section: 5 Boolean (n,2)-bent functions are almost balanced}.

 {We would also like to thank the anonymous reviewers for spotting some typos and comments that improved the presentation of the results. We in particular thank one reviewer for suggesting a simpler way to perform the computations in Theorem~\ref{thm:m4}}

The first author is supported by the National Science Foundation under grant number 2127742.

\begin{thebibliography}{10}
\providecommand{\url}[1]{{#1}}
\providecommand{\urlprefix}{URL }
\expandafter\ifx\csname urlstyle\endcsname\relax
  \providecommand{\doi}[1]{DOI~\discretionary{}{}{}#1}\else
  \providecommand{\doi}{DOI~\discretionary{}{}{}\begingroup
  \urlstyle{rm}\Url}\fi

\bibitem{Magma}
Bosma, W., Cannon, J., Playoust, C.: \href{https://doi.org/10.1006/jsco.1996.0125}{The {M}agma algebra system. {I}. {T}he
  user language}.
\newblock J. Symbolic Comput. \textbf{24}(3-4), 235--265 (1997).
\newblock Computational algebra and number theory (London, 1993)

\bibitem{Carlet2021_Book}
Carlet, C.: \href{https://doi.org/10.1017/9781108606806}{Boolean Functions for Cryptography and Coding Theory}.
\newblock Cambridge University Press (2021).

\bibitem{Carlet21}
Carlet, C.: \href{https://doi.org/10.1109/TIT.2021.3114958}{Bounds on the nonlinearity of differentially uniform functions by means of their image set size, and on their distance to affine functions}.
\newblock IEEE Transactions on Information Theory \textbf{67}(12), 8325--8334
  (2021).

\bibitem{carletding}
Carlet, C., Ding, C.: \href{https://doi.org/10.1016/j.jco.2003.08.008}{Highly nonlinear mappings}.
\newblock Journal of Complexity \textbf{20}(2), 205--244 (2004).
\newblock Festschrift for Harald Niederreiter, Special Issue on Coding and
Cryptography

\bibitem{CarletMesnager2011}
Carlet, C., Mesnager, S.: \href{https://doi.org/10.1016/j.jcta.2011.06.005}{On {D}illon's class {$\mathcal{H}$} of bent
  functions, {N}iho bent functions and o-polynomials}.
\newblock Journal of Combinatorial Theory, Series A \textbf{118}(8), 2392--2410
  (2011).

\bibitem{CM2013}
{\c{C}}e{\c{s}}melio{\u{g}}lu, A., Meidl, W.: \href{https://doi.org/10.1007/s10623-012-9686-2}{A construction of bent functions from plateaued functions}.
\newblock Designs, Codes and Cryptography \textbf{66}(1), 231--242 (2013).


\bibitem{CMP2021}
{\c{C}}e{\c{s}}melio{\u{g}}lu, A., Meidl, W., Pirsic, I.: \href{https://doi.org/10.1007/s10623-021-00919-y}{Vectorial bent functions and partial difference sets}.
\newblock Designs, Codes and Cryptography \textbf{89}(10), 2313--2330 (2021).

\bibitem{coulter1997planar}
Coulter, R.S., Matthews, R.W.: \href{https://doi.org/10.1023/A:1008292303803}{Planar functions and planes of {L}enz-{B}arlotti
  class II}.
\newblock Designs, Codes and Cryptography \textbf{10}(2), 167--184 (1997).

\bibitem{Coulter2014}
Coulter, R.S., Senger, S.: \href{https://doi.org/10.1007/s00026-014-0220-2}{On the number of distinct values of a class of functions with finite domain}.
\newblock Annals of Combinatorics \textbf{18}(2), 233--243 (2014).

\bibitem{Dillon1974}
Dillon, J.F.: \href{https://doi.org/10.13016/M2MS3K194}{Elementary {H}adamard difference sets}.
\newblock Ph.D. thesis, University of Maryland (1974).

\bibitem{ding2006family}
Ding, C., Yuan, J.: \href{https://doi.org/10.1016/j.jcta.2005.10.006}{A family of skew {H}adamard difference sets}.
\newblock Journal of Combinatorial Theory, Series A \textbf{113}(7), 1526--1535
  (2006).

\bibitem{dong2013note}
Dong, D., Zhang, X., Qu, L., Fu, S.: \href{https://doi.org/10.1016/j.ipl.2013.07.024}{A note on vectorial bent functions}.
\newblock Information Processing Letters \textbf{113}(22-24), 866--870 (2013).

\bibitem{golouglu2022exponential}
G{\"o}lo{\u{g}}lu, F., K{\"o}lsch, L.: \href{https://doi.org/10.1090/tran/8785}{An exponential bound on the number of non-isotopic commutative semifields}.
\newblock Transactions of the American Mathematical Society  (2022).

\bibitem{Hyun16}
Hyun, J.Y., Lee, J., Lee, Y.: \href{https://doi.org/10.1109/TIT.2016.2582217}{Explicit criteria for construction of plateaued functions}.
\newblock IEEE Transactions on Information Theory \textbf{62}(12), 7555--7565
  (2016).

\bibitem{Koelsch2023}
K{\"o}lsch, L., Kriepke, B., Kyureghyan, G.M.: \href{https://doi.org/10.1007/s10623-022-01094-4}{Image sets of perfectly nonlinear maps}.
\newblock Designs, Codes and Cryptography \textbf{91}(1), 1--27 (2023).

\bibitem{KUMAR1985}
Kumar, P., Scholtz, R., Welch, L.: \href{https://doi.org/10.1016/0097-3165(85)90049-4}{Generalized bent functions and their properties}.
\newblock Journal of Combinatorial Theory, Series A \textbf{40}(1), 90--107
  (1985).

\bibitem{KyureghyanP2008}
Kyureghyan, G.M., Pott, A.: \href{https://doi.org/10.1007/978-3-540-69499-1_10}{Some theorems on planar mappings}.
\newblock In: J.~von~zur Gathen, J.L. Ima{\~{n}}a, {\c{C}}.K. Ko{\c{c}} (eds.)
  Arithmetic of Finite Fields, pp. 117--122. Springer Berlin Heidelberg,
  Berlin, Heidelberg (2008).

\bibitem{kyureghyan2014inversion}
Kyureghyan, G.M., Suder, V.: \href{https://doi.org/10.1016/j.ffa.2013.10.002}{On inversion in {$\mathbb{Z}_{2^n- 1}$}}.
\newblock Finite Fields and Their Applications \textbf{25}, 234--254 (2014).

\bibitem{Meidl2022}
Meidl, W.: \href{https://doi.org/10.1007/s12095-022-00570-x}{A survey on {$p$}-ary and generalized bent functions}.
\newblock Cryptography and Communications \textbf{14}(4), 737--782 (2022).

\bibitem{Meier90}
Meier, W., Staffelbach, O.: \href{https://doi.org/10.1007/3-540-46885-4_53}{Nonlinearity criteria for cryptographic functions}.
\newblock In: J.J. Quisquater, J.~Vandewalle (eds.) Advances in Cryptology ---
  EUROCRYPT '89, pp. 549--562. Springer Berlin Heidelberg, Berlin, Heidelberg
  (1990).

\bibitem{Mesnager2016}
Mesnager, S.: \href{https://doi.org/10.1007/978-3-319-32595-8}{Bent Functions. Fundamentals and Results}. 1 edn.
\newblock Springer International Publishing (2016).

\bibitem{Nyberg91}
Nyberg, K.: \href{https://doi.org/10.1007/3-540-46877-3_13}{Constructions of bent functions and difference sets}.
\newblock In: I.B. Damg{\aa}rd (ed.) Advances in Cryptology --- EUROCRYPT '90,
  pp. 151--160. Springer Berlin Heidelberg, Berlin, Heidelberg (1991).

\bibitem{KaiNy91}
Nyberg, K.: \href{https://doi.org/10.1007/3-540-46416-6_32}{Perfect nonlinear {S}-{B}oxes}.
\newblock In: Advances in Cryptology - EUROCRYPT '91, Workshop on the Theory
  and Application of of Cryptographic Techniques, Brighton, UK, April 8-11,
  1991, Proceedings, \emph{Lecture Notes in Computer Science}, vol. 547, pp.
  378--386. Springer (1991).

\bibitem{PolPot_Dataset}
Polujan, A., Pott, A.: \href{https://doi.org/10.5281/zenodo.6983500}{{CCZ}-inequivalent quadratic vectorial {B}oolean bent
  functions in 8 variables}.
\newblock In: Zenodo Dataset (2022).

\bibitem{PolujanBFA22}
Polujan, A., Pott, A.: \href{https://boolean.w.uib.no/bfa-2022/}{Towards the classification of quadratic vectorial bent functions in 8 variables}.
\newblock In: The 7th International Workshop on Boolean Functions and their Applications (2022).

\bibitem{polujan2020design}
Polujan, A.A., Pott, A.: \href{https://doi.org/10.1109/TIT.2020.3040754}{On design-theoretic aspects of {B}oolean and vectorial bent functions}.
\newblock IEEE Transactions on Information Theory \textbf{67}(2), 1027--1037
  (2021).

\bibitem{Pott04}
Pott, A.: \href{https://doi.org/10.1016/S0166-218X(03)00293-2}{Nonlinear functions in abelian groups and relative difference sets}.
\newblock Discret. Appl. Math. \textbf{138}(1-2), 177--193 (2004).

\bibitem{Pott2016}
Pott, A.: \href{https://doi.org/10.1007/s10623-015-0151-x}{Almost perfect and planar functions}.
\newblock Designs, Codes and Cryptography \textbf{78}(1), 141--195 (2016).

\bibitem{ROTHAUS:1976}
Rothaus, O.: \href{https://doi.org/10.1016/0097-3165(76)90024-8}{On ``bent'' functions}.
\newblock Journal of Combinatorial Theory, Series A \textbf{20}(3), 300--305
  (1976).

\bibitem{Wang2023}
Wang, J., Fu, F.W.: \href{https://doi.org/10.1007/s10623-022-01103-6}{New results on vectorial dual-bent functions and partial
  difference sets}.
\newblock Designs, Codes and Cryptography \textbf{91}(1), 127--149 (2023).

\bibitem{weng2012further}
Weng, G., Zeng, X.: \href{https://doi.org/10.1007/s10623-011-9564-3}{Further results on planar {DO} functions and commutative semifields}.
\newblock Designs, Codes and Cryptography \textbf{63}(3), 413--423 (2012).

\bibitem{Mathematica}
{Wolfram Research, Inc.}: \href{https://www.wolfram.com/mathematica}{Mathematica, {V}ersion 13.2}.
\newblock Champaign, IL (2022).

\end{thebibliography}

\end{document}